\documentclass[11pt,a4paper]{article}
\usepackage{tikz}
\usepackage{amsfonts,amsthm,amsmath}
\usepackage{a4wide}
\usepackage{subcaption}
\usepackage{multirow}

\newtheorem{thm}{Theorem}
\newtheorem{prop}{Proposition}
\newtheorem{cor}{Corollary}
\newtheorem{lemma}{Lemma}

\providecommand{\keywords}[1]{\textbf{{Keywords: }} #1 \\}

\newcommand{\E}{\mathbb{E}}
\newcommand{\s}{\sigma}
\newcommand{\e}{{\rm e}}
\def \b {\beta}
\def \l {\lambda}
\def \m {\mu}

\def \d {\delta}
\def \a {\alpha}
\def \g {\gamma}
\def \k {\kappa}
\def \th {\theta}
\def \s {\sigma}
\def \t {\tau}
\def \e {\varepsilon}
\def \P {P}
\def \dd {{\rm d}}

\def \hx {\hat{x}}
\def \hy {\hat{y}}
\def \mui {\mu_\infty^\star}
\def \muT {\mu_T^\star}
\def \Xlm {X_{\mu}}
\def \Qlm {Q_{\mu}}
\def \Qm {Q_{\mu}}
\def \klm {\kappa_{\mu}}
\def \iy {\infty}
\def \Pih {\hat{\Pi}_{T}}
\def \muh {\hat{\mu}_T^\star}

\newcommand{\ed}{\,{\buildrel d \over =}\,}

\numberwithin{equation}{section}

\begin{document}

\title{Transient error approximation in a L\'evy queue}
\author{Britt Mathijsen\footnote{
Department of Mathematics and Computer Science, Eindhoven
University of Technology, P.O. Box 513, 5600 MB Eindhoven, The
Netherlands (b.w.j.mathijsen@tue.nl)} 
\and Bert Zwart\footnotemark[1]\ \footnote{Centrum Wiskunde \& Informatica, P.O. Box 94079,
1090 GB, Amsterdam, The Netherlands (bert.zwart@cwi.nl)} 
}
\date{}
\maketitle

\begin{abstract}
Motivated by a capacity allocation problem within a finite planning period, we conduct a transient analysis of a single-server queue with L\'evy input. From a cost minimization perspective, we investigate the error induced by using stationary congestion measures as opposed to time-dependent measures. Invoking recent results from fluctuation theory of L\'evy processes, we derive a refined cost function, that accounts for transient effects. This leads to a corrected capacity allocation rule for the transient single-server queue. Extensive numerical experiments indicate that the cost reductions achieved by this correction can by significant.
\end{abstract}
 
\keywords{Single-server queue, transient analysis, L\'evy processes, optimization.}

\section{Introduction}


The issue of matching a service system's capacity to stochastic demand induced by its clients arises in many practical settings. Typically, the resources available to satisfy demand are scarce and hence expensive. This forces the manager to consider a trade-off between the system efficiency and the quality of service perceived by its clients. In this paper, we focus on this trade-off in the context of the $M/G/1$ queue, in which the variable amendable for optimization is the server speed $\mu$.

In general, optimizing the server speed $\m$ in a single-server queue in time-homogeneous environment, while trading off congestion levels against capacity allocation cost, does not pose any technical challenges. Typically, the objective function to be minimized, the total cost function, has the shape
\begin{equation}\label{eq:intro}
\Pi_\iy(\mu) = \E[Q_\mu(\infty)] + \a\mu = \frac{\l\E[B^2] }{2(\mu-\l\E[B])} + \a\mu,
\end{equation}
where $\E[Q_\mu(\infty)]$ denotes the expected steady-state amount of work given server speed $\m$, and $B$ describes the service requirement per arrival. The parameter $\a>0$ represents the relative capacity allocation costs incurred by deploying service rate $\mu$. This one-dimensional optimization problem yields the optimizer
\begin{equation}
\mui = \lambda \E[B] + \sqrt{\frac{\l\E[B^2]}{2\a}}.
\end{equation}
Despite the simplicity and tractability of the problem described above, the presence of the \emph{steady-state} measure in the cost function in \eqref{eq:intro} should be handled carefully. By employing this particular cost structure, one automatically agrees with the underlying assumption of the system being sufficiently close to its steady state. 
However, referring the practical applications of the single-server model, system parameters rarely remain constant over time. Moreover, planning periods for the optimization problem are naturally finite. Hence, the \emph{true} expected costs incurred, which we denote by $\Pi_T(\mu)$, depend on the length of the planning period $T$. Consequently, the usage of steady-state models for decision making needs to be justified by a more elaborate time-dependent or \emph{transient} analysis for these type of settings.

The time-dependent behavior of the single-server queue received much attention in queueing theory. First efforts to analyze the time-dependent properties of the $M/G/1$ queue date back to the 1950s and 1960s, e.g \cite{Kendall1951,Takacs1955,Takacs1962,Gaver1959,Benes1957}. The analyses in these papers mostly yield implicit expressions for performance characteristics through Laplace transforms, integro-differential equations and infinite convolutions.
More specifically, there is vast literature on the transient analysis of the $M/M/1$ queue, with the goal to derive explicit expressions for queue length characteristics, see e.g. \cite{Prabhu1964,Cohen1969,Pegden1982,Abate1987}. 
These works provide a variety of explicit expressions for the transient dynamics, although the complexity of the resulting expressions, typically involving Bessel functions, expose the intricate intractability of the matter. Consequently, approximation methods for insightful quantification of the dynamics based on numerical \cite{Neuts1966} or asymptotic methods, have become prevalent in more recent literature.
The asymptotic methods either exploit knowledge on the evolution of the queueing process as time $t$ grows large \cite{Newell1982,Odoni1983,Abate1987}, or the arrival rate $\l$ is increased to infinity \cite{Gaver1968,Abate1987a,Abate1987b}. 
It is noteworthy that a substantial contribution to the transient literature is made by Abate and Whitt \cite{Abate1987a,Abate1987b,Abate1987,Abate1994} who exploit the existence of a decomposition of the mean transient queue length and obtain expressions for the moments of the queue length and virtual waiting through probabilistic arguments in several queueing models. 
More recently, asymptotic methods have been used to justify the application of stationary performance measures in Markovian environments or to refine them, see e.g. \cite{Green1991,Whitt1991}.
Other approximative methods under the name of uniform acceleration expansions \cite{Massey1998} have been developed to reveal the asymptotic behavior of the single-server queue as a function of $t$, which are moreover able to capture time-varying arrival rates.


The majority of the works mentioned above do reflect on the error imposed by usage of steady-state performance metrics instead of the correct time-dependent counterpart. However, no light has been shed on the accumulation of this error over a finite period of time. To the best of our knowledge, the only work that addresses this issue is the paper by Steckley and Henderson \cite{Steckley2007}, who compute an approximation for the error accumulated between the steady-state and transient delay probability. Our analysis on the other hand is centered around the mean workload, which requires a different approach. In addition, the focus in \cite{Steckley2007} is on performance measures only, while the main goal of our paper is to investigate the quality of staffing rules. 

Although the $M/G/1$ queue serves as the leading example in our analysis, we choose to use a more general framework for the arrival process of the queue. Namely, we let the server face a L\'evy process.
This gives the advantage that once we have obtained the results, we can apply them to broader queue input classes, such as Brownian motion and the Gamma process.


To shed light on the influence of the transience of the queueing process on traditional staffing questions, we will study the capacity allocation problem in the context of cost minimization in which the objective function is $\Pi_T(\mu)$, i.e. a function of both $\mu$ and $T$. We investigate how the invalidity of the stationary assumption is echoed through the operational cost accounting for congestion-related penalties. 

Furthermore, we establish a result on the strict convexity of the function $\Pi_T(\mu)$, for almost all values of $T$ (with a few minor exceptions for certain deterministic initial states), which is an essential property for convergence of both cost function and corresponding minimizer to their stationary counterparts.

As it will appear that an exact analysis of this disparity is intractable, we will present an explicit approximate correction to the conventional stationary objective function given by $\Psi(\mu)/T$ and prove that
\begin{equation}
\Pi_T(\mu) = \Pi_\iy(\mu) + \frac{\Psi(\mu)}{T} + O(1/T^2),
\end{equation}
with the help of recent results from the fluctuation theory of L\'evy processes. 
 Based on this refinement we ultimately examine how incorporating transient effects reflects in setting the optimal capacity level and propose a refinement to the steady-state capacity allocation rule, 
\begin{equation}
\muT = \mui + \frac{\mu_\bullet}{T} + o(1/T).
\end{equation}
We moreover deduce an explicit expression for $\mu_\bullet$ in terms of the initial state and the first three moments of the service requirement per arrival.
It is noteworthy that similar refined square-root staffing rules have been proposed for multi-server queues in the Halfin-Whitt regime, see e.g. \cite{Janssen2008,Janssen2011,Janssen2015,Randhawa2014,Zhang2012}. In those cases, the relevant decision value is the number of servers  and refinements are derived for $\l\to\iy$, whereas we consider the regime $T\to\infty$. 

  Building upon the insights gained through the analysis of this optimality gap, we reflect on the parameter settings of the underlying queueing process in which our refined capacity sizing rule yields significant improvement and in which cases it has little effect. Special emphasis is put on the relationship between the accuracy of the standard procedure and the length of the planning period.

The remainder of the paper is structured as follows. Section 2 is devoted to the model description and presents some preliminary results. The main result will be given in Section 3 and results regarding the optimization problem will be discussed in Section 4, followed by the validation of our novel techniques through numerical experiments in Section 5. We will give some concluding remarks and topics for further research in Section 6. We have deferred all proofs to the Appendix.

\section{Model description}

\subsection{A queueing model with L\'evy input \label{sec:levymodel}}
The model that inspired our study is the standard $M/G/1$ queue starting out of equilibrium. Customers arrive to the queue according to a Poisson process with rate $\l$ and each arrival has service requirement $B_i$, stemming from a common random variable $B$. 
Without loss of generality we will assume $\E[B] = 1$ throughout. The server is able to remove $\mu$ amounts of work from the system per time unit; a variable we will refer to as the \emph{server speed}. 
E.g. if $\mu = 3$ and two customers are in the system with remaining service times $4$ and $2$, then  the queue will be empty 2 time units later, provided that no new arrivals occur in the meantime.
Let $N_\l(t)$ denote the number of arrivals until time $t$.
Accordingly, the total work generated by the customers is given by
\begin{equation}
Z_\l(t) = \sum_{i=1}^{N_\l(t)} B_i.
 \end{equation} 
Furthermore, define $\Xlm(t) = Z_\l(t) - \mu t$. We call $\Xlm$ the \emph{net-input process}. 
More generally, we assume throughout the paper that $\Xlm$ is a L\'evy process.
Specifically, we let $Z_\l$ be of the form $Z_\l(t) = U(\l t)$, where $U$ is a spectrally positive L\'evy process generated by the triplet $(a,\s,\nu)$ and $\E[U(1)] = 1$. 
This restriction to spectrally positive processes is equivalent to stating $\nu(-\infty,0)=0$ and is a vital assumption to our analysis. 
Subsequently, we assume the net-input process $\Xlm$ to be
\begin{equation}
\label{eq:Xlmprocess}
\Xlm(t) = U(\l t) - \mu t, \qquad t \geq 0.
\end{equation}
Note that by setting $a=\s=0$ and $\nu = \l\, F_B$, where $F_B$ is the cumulative distribution function of $B$, we recover the original $M/G/1$ queue. 
The stochastic process central to our analysis is the \emph{workload process} $Q_{\mu}(t)$, $t\geq 0$, which describes the amount of work the server is facing at time $t$. 
The net-input process $\Xlm$ completely determines the trajectory of $Q_{\mu}$, namely
\begin{equation}\label{eq:Qlm}
Q_{\mu}(t) = \max\{ Q(0) + \Xlm(t), \sup_{s\in[0,t]} [\Xlm(t)-\Xlm(s)]\}, \qquad t\geq 0,
\end{equation}
where $Q(0)$ is the initial workload in the system. 
In fact, $Q_{\mu}$ is the reflected version of $\Xlm$ with reflection barrier at zero.
Careful inspection of the structure also reveals that $X_{\m}(t) \equiv X_{\l/\mu,1}(\mu t) \equiv X_{1,\mu/\l}(\l t)$, so that
\begin{equation}
\label{eq:Qidentity}
Q_{\m}(t) \ed Q_{\l/\mu,1}(\mu t) \ed Q_{1,\mu/\l}(\l t)
\end{equation}
for all $\l,\m,t>0$. 
This identity will prove to be convenient for numerical analysis in Section \ref{sec:numerics}.

The process $Q_{\m}$ is a natural indicator of the level of congestion in the system and therefore a good choice for quantifying the Quality of Service (QoS) received by a client.
We remark that alternative processes characterizing congestion in the system can be directly deduced from $Q_{\m}(t)$. For example, consider the virtual waiting time process $V_{\mu}(t)$, which is the waiting time a customer would experience if he arrives at time $t$. This satisfies the relation $V_{\mu}(t) \equiv Q_{\m}(t)/\mu$ for all $t\geq 0$.
Likewise, the expected number of the customers in the system $L_{\m}(t)$ at time $t\geq 0$ is given by Little's law
\begin{equation}
\E[L_{\m}(t)] = \l\, \E[V_{\m}(t)] = \frac{\l}{\mu}\, \E[Q_{\mu}(t)].
\end{equation}
To facilitate our investigation of the queueing model, we end this subsection by introducing some notation regarding the net-input and workload process and by stating a useful preliminary result concerning the stationary process $\Qlm(\iy)$. 
Throughout the paper we assume $\mu>\l$ to ensure ergodicity of the queue and existence of the limit
\begin{equation}
\Qlm(\iy) := \lim_{t\to\iy} \Qlm(t),
\end{equation}
for any initial state $Q(0)$. This random variable necessarily coincides with the stationary distribution of $\Qlm(t)$. 
By $\k_U(\cdot)$ and $\k_{\mu}(\cdot)$ we denote the L\'evy exponents of the processes $U$ and $\Xlm$, respectively:
\begin{equation}
\k_{\m}(\th) = \log \E[e^{\th \Xlm(1)}] = \log \E[e^{\th(U(\l) - \mu)}] = \l \k_U(\th) - \mu \th.
\end{equation}
Furthermore, define $u_k = \E[\{U(1) - \E U(1)\}^k]$ for $k=2,3,...$. 
Using this representation we obtain the following preliminary result.
\begin{lemma}\label{lemma:workloadmoments}
Let $\E|U(1)|<\infty$, $u_2, u_3 < \iy$ and $\mu > \l$. If $Q_{\mu}(\infty)$ represents the steady-state distribution of the workload process, then
\begin{equation}
\E[\Qlm(\infty)] = \frac{\l u_2}{2(\mu-\l)},\qquad \E[Q_{\mu}^2(\iy)]=\frac{\l^2u_2^2}{2(\mu-\l)^2} + \frac{\l u_3}{3(\mu-\l)}.
\end{equation}
\end{lemma}

\subsection{Finite horizon}

For the purpose of this paper, we are interested in the dynamics of the workload process within a fixed time frame of length $T>0$. 
For all $0\leq t \leq T$, we assume that the parameters of the queue, $\l,\m,u_2,u_3$, remain unchanged. 
If at $t=0$ the queue is not in steady-state corresponding to the specified parameters of the starting period, the process $\{ \Qlm(t)\,:t\in[0,T] \}$ differs from its stationary counterpart $\Qlm(\infty)$. 
To illustrate this, Figure \ref{fig:transientmeans} depicts the expected value $\Qlm$ in a $M/M/1$ queue as a function of time for several initial workloads $Q(0)$ for a particular setting of $\l$ and $\m$. 
Clearly, transient behavior of $\E[\Qlm(t)]$, for $Q(0) \neq \Qlm(\iy)$, differs significantly from the steady-state mean with the same system parameters. 
Note that even if $Q(0) \equiv \E[\Qlm(\iy)]$, the time-dependent mean does not coincide with the steady-state mean. Moreover, $\E[\Qlm(t)]$ is not even a strictly increasing nor decreasing function of time. This phenomenon is a consequence of the decomposition of the transient mean into one strictly increasing, and a strictly decreasing term for $Q(0)>0$, as was studied in \cite{Abate1987}.
Nonetheless, $\Qlm(t)$ converges in distribution to $\Qlm(\infty)$ as $t\to\iy$, if $\mu>\l$. 
 
\begin{figure}
\centering
\begin{tikzpicture}[xscale=0.2,yscale=0.3]
\draw (0,0) -- coordinate (x axis mid) (50,0);
    \draw (0,0) -- coordinate (y axis mid) (0,21);
	\node[right] at (51,0) {$t$};
	\node[rotate=90, above=0.7 cm] at (y axis mid) {$\E[\Qlm(t)]$};
	
	\draw[dashed, thick, gray] (0,10) -- coordinate (eq) (51,10);
	
	\definecolor{col1}{rgb}{0.368417, 0.506779, 0.709798}
 	\definecolor{col2}{rgb}{0.880722, 0.611041, 0.142051}
 	\definecolor{col3}{rgb}{0.560181, 0.691569, 0.194885}	
 	\definecolor{col4}{rgb}{0.922526, 0.385626, 0.209179}
	
	\draw[->] (24,6.4) --  coordinate (a1) (21.65,8.49574);
	\node[right=0.6cm,below=0.3cm] at (a1) {$Q(0)\equiv 0$};
	\draw[->] (14,6.2) --  coordinate (a2) (13.,8.49434);
	\node[right=0.2cm,below=0.3cm] at (a2) {$Q(0)\equiv10$};
	\draw[->] (9,15.5) --  coordinate (a3) (7.5,13.7648);
	\node[right=0.6cm,above=0.1cm] at (a3) {$Q(0)\equiv20$};
	\draw[->] (40,12.5) -- coordinate (a4) (38,10.9712);
	\node[right,above=0.3cm] at (a4) { $Q(0)\sim \exp\left(\tfrac{1}{15}\right)$ };

    	\foreach \x in {0,10,...,50}
     		\draw (\x,1pt) -- (\x,-10pt)
			node[anchor=north] {\x};
    	\foreach \y in {5,10,15,20}
     		\draw (1pt,\y) -- (-20pt,\y)
     			node[anchor=east] {\y};

	\draw[thick,color = col1] plot
			file { means0.txt};
	\draw[thick,color = col2] plot
			file { means10.txt};
    \draw[thick,color = col3] plot
			file { means20.txt};
	\draw[thick,color = col4] plot
			file { meansExp.txt};		
\end{tikzpicture}
\caption{Time-dependent mean workload in $M/M/1$ queue with $\l = 10$ and server speed $\mu=11$ for different initial states $Q(0)$. The dashed line depicts $\E\Qlm(\iy)$.}
\label{fig:transientmeans}
\end{figure}
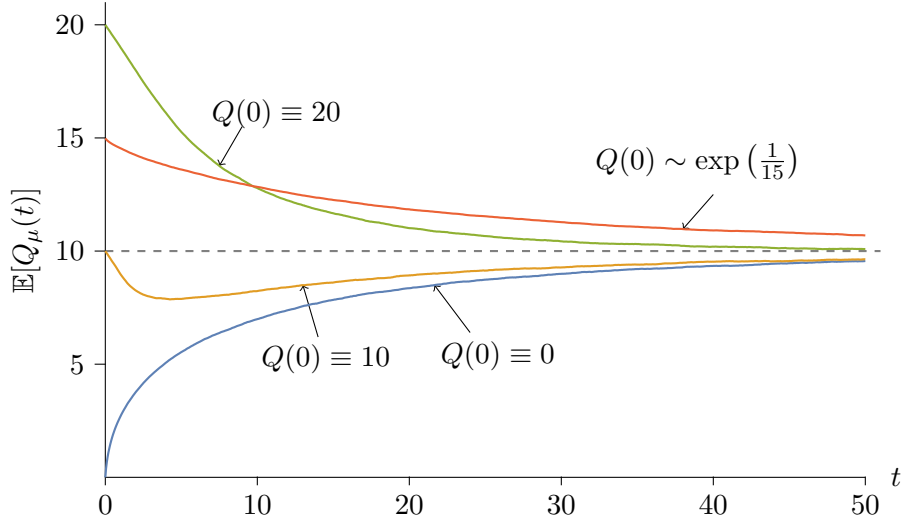
Since the time horizon of our analysis is limited to $t\leq T$, the process may not approach the steady-state distribution sufficiently close to appropriately use its steady-state properties for capacity allocation. 
To overcome this disparity, we propose a way to include the influence of this transient phase in the capacity allocation problem.
 
\subsection{Cost structure}

As mentioned before, we are interested in balancing the QoS and efficiency of the queue by choosing the optimal server speed $\mu$. 
The adjective \emph{optimal} indicates that we intend to choose the speed according to some objective function. 
In our case, we conduct our analysis based on a cost function, which consists of a part accounting for the penalty for congestion in the system and a part for staffing cost. The cost value of both parts is governed by the variable $\mu$. 
The instantaneous cost incurred at time $t$ equals
\begin{equation}
\E[\Qlm(t)] + \a \mu,
\end{equation}
where $\a$ is a positive constant defining the \emph{relative staffing cost}. 
Hence, the cost structure we apply is a combination of the transient mean of the workload process and a linear staffing cost. 
Accumulated and normalized over the period $[0,T]$, the cost function on which the rest of this paper will be based equals
\begin{equation}\label{eq:PiT}
\Pi_{T}(\mu) := \frac{1}{T}\int_0^T\left(  \E[\Qlm(t)] + \a\mu\, \right) \dd t 
 = \frac{1}{T} \int_0^T \E[\Qlm(t)]\, \dd t + \a\mu.
\end{equation}
We use shorthand notation for the normalized congestion costs:
\begin{equation}\label{eq:CTmu}
C_{T}(\mu) := \frac{1}{T}\int_0^T \E[Q_{\mu}(t)] \dd t,
\end{equation}
and $C_{\iy}(\mu) = \E[\Qlm(\iy)]$. 
In order to compare the actual costs incurred over the interval $[0,T]$ to the cost function of the queue in stationary conditions, we define
\begin{equation}\label{eq:PiInf}
\Pi_{\iy}(\mu) := C_{\iy}(\mu) + \a \mu = \E[Q_\mu(\iy)] + \a\mu,
\end{equation}
which allows an explicit expression by Lemma \ref{lemma:workloadmoments}. 
Also, note that by dominated convergence theorem
\begin{equation}
\lim_{T\to\iy} \Pi_{T}(\mu) = \Pi_{\iy}(\mu),
\end{equation} 
for all $\mu$.
Rewriting \eqref{eq:PiT} gives the relation
\begin{align}
\Pi_{T}(\mu) &= \frac{1}{T}\int_0^{T} \left( \E[\Qlm(t)] - \E[\Qlm(\iy)] \right)\, \dd t + \E[\Qlm(\iy)] + \a\mu =  \Omega_{T}(\mu) + \Pi_{\infty}(\mu).
\label{eq:decomp}
\end{align}
Section \ref{sec:analysis} is concerned with the analysis of the correction factor $\Omega_{T}(\mu)$.

Ultimately, we are concerned with the additional costs incurred by choosing the server speed through minimization of $\Pi_{\iy}(\mu)$ instead of $\Pi_{T}(\mu))$. 
Therefore, we formulate the exact and approximate optimization problems as follows
\begin{equation}\label{eq:muStar}
\mu_T^\star := \arg\min_{\mu\geq 0} \Pi_{T}(\mu), \qquad \qquad \mu_\infty^\star := \arg\min_{\mu\geq 0} \Pi_{\iy}(\mu),
\end{equation}
\begin{equation}\label{eq:piStar}
\Pi_{T}^\star = \Pi_{T}(\mu_T^\star), \qquad \qquad \Pi_{\iy}^\star = \Pi_{\iy}(\mu_\iy^\star).
\end{equation}
In Section \ref{sec:optimization} we turn to the comparison of $\mu_T^{\star}$ and $\mu_\iy^\star$ as well as the \emph{optimality gap} $\Pi_{\iy}^\star - \Pi_{T}^\star$.
For sake of clarity, we omit the subscript $\l$ in our expressions if no ambiguity is possible.

\section{Analysis of the objective function \label{sec:analysis}}
From \eqref{eq:decomp} it is evident that, for finding an explicit characterization of $\Pi_{T}(\mu)$, it suffices to study the term $\Omega_T(\mu)$ in more detail. We start by stating the main result of this section, which describes the leading order behavior of $\Omega_T(\mu)$ as $T$ increases.
\begin{thm}
Let $X_\mu(t)$ be of the form \eqref{eq:Xlmprocess}. If $\E[Q(0)^2], \E[Q(0)^3] < \iy$ and $u_2,u_3 < \iy$, then
\begin{equation}
\Omega_T(\mu) = \frac{1}{2T(\mu-\l)}\left( \E[Q(0)^2] - \frac{\l^2 u_2^2}{2(\mu-\l)^2} - \frac{\l u_3}{3(\mu-\l)}\right) + O\left(\frac{1}{T^2}\right),
\end{equation} 
for $\mu>\l$.
\end{thm}
Note that this expression provides an \emph{approximation} of the actual cost function $\Pi_T(\mu)$. We elaborate on the implications of this additional information on the optimization problem in Section 4. 

In the remainder of this section we provide a detailed description of the steps taken to obtain this outcome. Proofs of the intermediate results can be found in Appendix A.
\subsection{Constructing a coupling}

Before starting our analysis with the correction term $\Omega_{T}(\mu)$ we introduce some auxiliary notation. 
By $Q_\mu^A(t)$ we denote the workload process as described in Subsection \ref{sec:levymodel} with $Q(0)\ed A$ and $\E_A$ the expectation with respect to the random variable $A$. 
To be able to compare $\E[\Qm^Z(t)]$ and $\E[Q_\mu(\iy)]$ as in $\Omega_T(\mu)$, we will use a coupling technique.
For brevity, denote by $Z$ a random variable for which $Z\ed \Qm(\iy)$. Then $\Qm(\iy) \ed \Qm^{Z}(t)$ for all $t \geq 0$ and $\E[\Qm(\iy)] = \E_Z[\Qm^{Z}(t)]$. 
Hence, quantifying the difference between the transient and stationary mean is equivalent to comparing the workload processes of two queues starting in two different (random) states at $t=0$. 
For now, assume $Q(0)\equiv x \geq 0$. Later, we relax this by replacing $x$ by the random variable $Q(0)$. In this subsection, we will omit the subscript $\mu$ for brevity.

Equation \eqref{eq:Qlm} shows that all randomness in $Q$ originates from the process $X(t)$. 
With this in mind, we couple the processes $Q^x(t)$ and $Q^{Z}(t)$ on a sample path level by feeding both queues the same net-input process $X(t)$ for $t\geq 0$. 
This allows us to compare the processes in the same probability space, 
\begin{align}
\E[Q^x(t)] - \E[Q(\infty)] &= \E_{X}[Q^x(t)] - \E_{X}\hspace{-4pt}\left[\E_Z[Q^{Z}(t)]\right]\nonumber\\
 &= \E_{Z}\left[\E_{X}\hspace{-4pt}\left[ Q^x(t) - Q^{Z}(t)\right]\right].
\end{align}
For brevity, we also replace $\Qm(\iy)$ by the variable $y$. At the end of our analysis we will obtain the original form by randomization. 
Define
\begin{equation}
Y^{x,y}(t) := Q^x(t) - Q^y(t).
\end{equation}
Then 
\begin{equation}
\Omega_{T}^{x,y} := \frac{1}{T}\,\int_0^T \E\left[Y^{x,0}(t)\right] \, \dd t
\end{equation}
and
\begin{equation}
\Omega_{T} = \E_{Q(\iy)}\left[ \Omega_T^{x,\Qm(\iy)} (\mu) \right].
\end{equation}
A possible sample path triple for $Q^x(t)$, $Q^0(t)$ and $Y^{x,y}(t)$ is depicted in Figure \ref{fig:samplePaths}. As we see from this figure, $Y^{x,y}(t)$ has nice structural properties which we will exploit in the next subsection. 

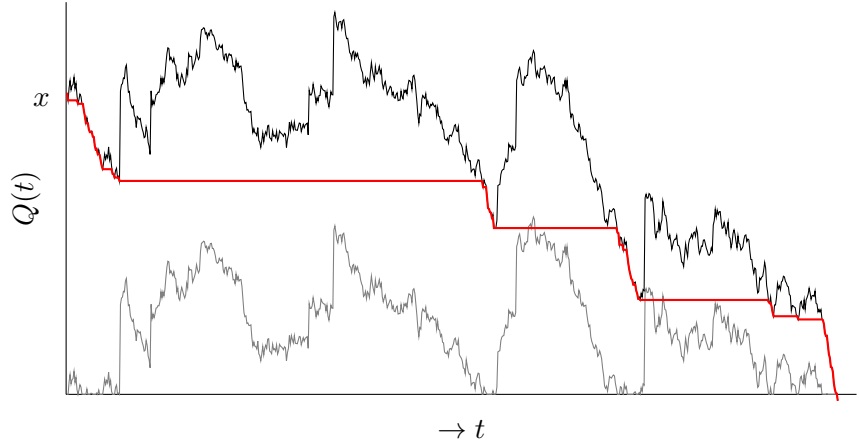
\begin{figure}
\centering
\begin{tikzpicture}[y=0.8cm, x=0.013cm]
	\draw (0,0) -- coordinate (x axis mid) (800,0);
    \draw (0,0) -- coordinate (y axis mid) (0,6.5);
	\node[below=0.2cm] at (x axis mid) {$\to t$};
	\node[rotate=90, above=0.2cm] at (y axis mid) {$Q(t)$};
	\node[above=1.3cm,left =0.08 cm] at (y axis mid) {$x$};

	\draw plot
			file { samplePathLevy.txt};
	\draw[color = gray] plot
			file { samplePathLevy2.txt};
	\draw[thick,color=red] plot
			file { runningMinimumLevy.txt};
\end{tikzpicture}
\caption{Sample path visualization of the processes $Q^x(t)$ (solid), $Q^0(t)$ (gray) and $Y^{x,0}(t)$ (red).}
\label{fig:samplePaths}
\end{figure}

\subsection{Difference process and leading order behavior of the correction term}

We further examine the \emph{difference process} $Y^{x,y}(t)$. Let us assume that $x>y$. Recall from \eqref{eq:Qlm},
\begin{equation}\label{eq:Wz}
Q^z(t) = \max\{ z + X(t),\, \sup_{0<s\leq t} [X(t)-X(s)]\} = X(t) + \max\{ z, -\inf_{0\leq s\leq t} X(s)\},
\end{equation}
where $X(t)$ is a L\'evy process with no negative jumps.
Let $\tau^x(z)$, $z<x$ denote the first passage time of level $z$ by the process $Q^x$, i.e.
\begin{equation}
\tau^x(z) := \inf \left\{ t \geq 0\, |\, Q^x(t) < z \,\right\}.
\end{equation}
Then it is easily seen that
\begin{equation}
Q^z(t) = \left\{
\begin{array}{ll}
z + X(t), & {\rm if }\ t <\tau^z(0), \\
\sup_{0<s\leq t} [X(t)-X(s)], & {\rm if }\ t \geq \tau^z(0).
\end{array}\right.
\end{equation}
Consequently,
\begin{equation}\label{eq:Yxy}
Y^{x,y}(t) = \left\{
\begin{array}{ll}
x - y, & \text{if }t < \tau^y(0),\\
\inf_{0<s\leq t} \{ x+X(s)\}, & \text{if }\tau^y(0) \leq t < \tau^x(0),\\
0, & \text{if }\tau^x(0) \leq t.
\end{array}\right.
\end{equation}
Using this representation we can identify
\begin{equation}
\Omega^{x,y}_T = \frac{1}{T}\,\E\left[\int_0^{\tau^x(0)\wedge T} Y^{x,y}(t) \dd t\right],
\end{equation}
where $\wedge$ denotes the minimum operator, due to the fact $Y^{x,y}(t) = 0$ for $t\geq \tau^x(0)$. 
Subsequently, we decompose $\Omega_T^{x,y}$ into two terms 
\begin{equation}
\Psi^{x,y}_T := \frac{1}{T} \int_0^\infty \E[Y^{x,y}(t)]\, \dd t  \qquad
\text{and}
\qquad
\Delta_T^{x,y} := \Omega_T^{x,y} - \Psi_T^{x,y}.
\label{eq:Deltaxy}
\end{equation}
Note that $\Psi_T^{x,y}$ is obtained by replacing $T$ by $\infty$ only in the integration bound.
This decomposition is insightful, because $\Psi_T^{x,y}$ prescribes the leading order behavior of $\Omega_T^{x,y}$, while $\Delta_T^{x,y}$ captures the smaller order error term. 
In this section, we only consider $\Psi_T^{x,y}$. Subsection \ref{sec:trunc} investigates the magnitude of $\Delta_T^{x,y}$.
The next preliminary result presents a useful property of $\Psi_T^{x,y}$.
\begin{lemma}
Let $x>y$. If $\E[\tau^x(0)]<\iy$, then
\begin{equation}\label{eq:H(x,y)}
\Psi^{x,y}_T = \frac{1}{T}\,\E[\tau^{y}(0)](x-y) + \Psi^{x-y,0}_T.
\end{equation}
\end{lemma}
This leaves us with two unknowns $\E[\t^y(0)]$ and $\Psi_T^{x-y,0}$.
The next lemma gives an equivalent form for the latter.
\begin{lemma}
For $z\geq 0$ and $\E[\tau^z(0)] < \iy$,
\begin{equation}\label{eq:H(x,0)}
\Psi^{z,0}_T = \int_0^z \E[\tau^y(0)]\, \dd y.
\end{equation}
\end{lemma}

As the term $\E[\tau^y(0)]$ appears in many of the preliminary results, we devote attention to this in the next subsection.

\subsubsection*{First passage time}
When studying the first passage time $\tau^x(y)$ of the workload process starting $x$, we first observe that $\{\tau^x(z-y)\}_{y=0}^x$ is a spectrally positive L\'evy process itself. 
More precisely, it is a subordinator, i.e. a L\'evy process whose paths are almost surely non-decreasing \cite{Kyprianou2006}. 
In order to calculate $\E[\tau^x(x-y)]$ we use theory presented in \cite[Section 46]{Sato1999}, although results presented there are valid for spectrally \emph{negative} L\'evy processes, as opposed to the absence of negative jumps in our case. 
Nonetheless, our setting is easily transformed into this framework by observing that $\hat{X} \equiv -X$, that is $\hat{X}(t) = -X(t)$ for all $t\geq 0$, is spectrally negative. 
Furthermore, let
\begin{equation}
\label{eq:transformedTau}
\hat{\tau}^0(y) := \inf\{ t \geq 0\,:\, \hat{X}(t) > y\} =  \inf\{ t \geq 0\,:\, x+X(t)< x-y\} = \tau^x(x-y).
\end{equation}
For completeness, we cite \cite[Thm~46.3]{Sato1999}.
\begin{thm}
Let $\hat{X}(t)$ be a spectrally negative L\'evy process with generating triplet $(-a,\s,\hat{\nu})$ and $\hat{\tau}^0(y)$ its corresponding hitting time process. Define $\Upsilon(\th)$ for $\th\geq 0$ as
\begin{equation}\label{eq:thmCharExp}
\Upsilon(\th) = -a\th + \tfrac{1}{2}\s^2\th^2 + \int_{-\infty}^0 (e^{\th x}-1-\th x{\bf 1}_{[-1,0)}(x))\, \hat{\nu}(\dd x).
\end{equation}
Then $\Upsilon(\th)$ is strictly increasing and continuous, $\Upsilon(0)=0$, and $\Upsilon(\th)\to\infty$ as $\th\to\infty$. For $x\geq 0$ and $0\leq u < \infty$ we have
\begin{equation}\label{eq:invCharExp}
\E[\exp(-u\hat{\tau}^0(y))] = \exp(-y\,\Upsilon^{-1}(u)),
\end{equation}
where $\th=\Upsilon^{-1}(u)$ is the inverse function of $u=\Upsilon(\th)$. 
\end{thm}
This immediately induces an expression for $\Psi^{z,0}$.
\begin{cor}\label{cor:Psixy}
Let $X(t)$ be a spectrally positive L\'evy process defined as in \eqref{eq:Xlmprocess} with $\mu > \l$. Let $\Psi^{z,0}_T$ as in \eqref{eq:H(x,0)}. Then 
\begin{equation}
\Psi^{z,0}_T = \frac{z^2}{2T(\mu-\l)}.
\end{equation}
Furthermore, if $x,y\geq 0$,
\begin{equation}\label{eq:mainResult}
\Psi^{x,y}_T = \frac{x^2-y^2}{2T(\mu-\l)}.
\end{equation}
\end{cor}

\subsubsection*{Randomization}
As we stated before, we easily obtain the original $\Omega_T$ from $\Omega_T^{x,y}$ through substitution of $x$ and $y$ by $Q(0)$ and $Q(\iy)$, respectively, and taking the expectation. 
In the previous paragraph, we deduced an explicit expression for $\Psi_T^{x,y}$, the leading order term for $\Omega_T^{x,y}$. 
Therefore we equivalently get an approximation for $\Omega_T$, given by
\begin{equation}
\Psi_T := \frac{1}{T} \int_0^\iy \left( \E[Q(t)]-\E[Q(\iy)] \right)\, \dd t,
\end{equation}
through randomization of $x$ and $y$ in $\Psi_T^{x,y}$.
By combining the results in Corollary \ref{cor:Psixy} and Lemma \ref{lemma:workloadmoments} we directly prove the result in Theorem 1.

\subsection{Truncation error}\label{sec:trunc}

In order to get a better comprehension of the properties of $\Psi_T$, we depict the value in terms of the (infinite) region between the curves $\E[Q(t)]$, $\E[Q(\iy)]$ and the vertical axis for the case $Q(0)\equiv 0$ in Figure \ref{fig:PsiVisualization}.
In this figure, $\Omega_T$ is given by the area enclosed by the two curves, the vertical axis and the line $t=T$. 
One can see that the main contribution to the correction term $\Omega_T$ is given for small $t$. 
As $t$ increases, the difference between transient and stationary mean decreases.
Hence for moderate values of $T$, the contribution to the integral in \eqref{eq:Deltaxy} is only minor compared to the contribution over the interval $[0,T]$. 

\begin{figure}
\centering
\begin{tikzpicture}[xscale=0.13,yscale=0.3]
	\node[below=0.4cm,right=0.5cm] at (x axis mid) {$\to t$};
	
	\draw[dashed, thick, fill =gray!30] (0,0) rectangle coordinate (eq) (50,10);
	\node[] at (-7,10) {$\E[Q(\infty)]$};
	\node[] at (-3,0) {$0$};

	\draw[->] (18,6.4) coordinate (a1) --   (21.65,8.49574);
	\node[below] at (a1) {$x=0$}; 

    	\foreach \x in {30}
     		\draw (\x,1pt) -- (\x,-10pt)
			node[anchor=north] {$T$};
    	\foreach \y in {10}
     		\draw (1pt,\y) -- (-20pt,\y);

	\draw[thick,color = gray,fill=white] plot
			file {means0_2.txt};
		
	\draw[thick] (0,0) -- coordinate (x axis mid) (50,0);
    \draw[thick] (0,0) -- coordinate (y axis mid) (0,12);		
    \draw[color=white,very thick] (50,0.05) -- (50,9.56);
    
    \draw[very thick, dotted] (30,0) -- (30,10);
    
    \draw[->] (18,6.8) coordinate (delta) -- (17,8.7);
    \node[below] at (delta) {$\Psi_{T}$};
    
    \draw[->] (38,8.1) coordinate (delta) -- (36,9.7);
    \node[below] at (delta) {$\Delta_{T}$};
\end{tikzpicture}
\caption{Visualization of $\Omega_T$ and $\Psi_T$ as the area between the curves $E[Q(t)]$, $\E[Q(\iy)]$ for $Q(0) = 0$.}
\label{fig:PsiVisualization}
\end{figure}
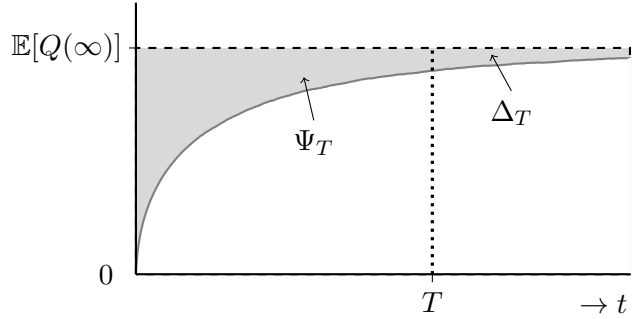

Recall the definition of $\Delta^{x,y}_T$ as in \eqref{eq:Deltaxy}. As we eluded to in Subsection 3.2
we claim the contribution of $\Delta^{x,y}_T$ to $\Omega_T^{x,y}$ is negligible compared to $\Psi^{x,y}_T$. Also note that
\begin{equation}
\label{eq:Delta}
\Delta_T := \Omega_T - \Psi_T = {-}\frac{1}{T} \int_T^\iy \E[Q(t)] - \E[Q(\iy)]\,\dd t.
\end{equation}
can be derived through $\Delta^{x,y}_T$ in a similar manner as we did for $\Psi^{x,y}_T$ to obtain $\Psi_T$.
To substantiate our claim, we compute an upper bound for $\Delta^{x,y}_T$ of order $1/T^2$. The existence of such an upper bound poses a limit on the error this tail integral contributed to the cost structure as a whole. 

\begin{prop} \label{prop:truncation_error}
Let $x,y\geq 0$ and $\E[\max\{Q(0),Q_\m(\iy)\}^3] < \iy$. Then
\begin{equation}
|\Delta^{x,y}_T| \leq \frac{1}{T^2}\left(\frac{\max(y,x)^3}{3(\mu-\l)^2}+\frac{u_2 \max(y,x)^2}{2(\m-\l)^3}\right) 
\end{equation}
and
\begin{equation}
|\Delta_T| \leq \frac{1}{T^2}\left(\frac{\E[\max(Q(0),Q_\m(\iy))^3]}{3(\mu-\l)^2}+\frac{u_2 \E[\max(Q(0),Q_\m(\iy))^2]}{2(\m-\l)^3}\right).
\end{equation}

\end{prop}

\section{Optimization \label{sec:optimization}}
The result in Theorem 1, characterizing the leading order behavior of $\Omega_T(\mu)$, also reveals the behavior of $\Pi_T(\mu)$ in leading order. Namely, 
\begin{equation}
\Pi_T(\mu) = \Pi_\iy(\mu) + \Psi_T(\mu) + O(1/T^2).
\end{equation}
In fact, this representation naturally gives rise to an \emph{approximation} of the actual cost function:
\begin{align}\label{eq:decomposition}
\hat{\Pi}_{T}(\mu) := \Pi_{\iy}(\mu) + \Psi_T(\mu) 
\end{align}
(We again include $\mu$ in the descriptions of variables derived in previous sections, because of the central role this decision variable will be playing within this section.)
Denote the corresponding minimizer of $\Pih$ by
\begin{equation}\label{eq:muhat}
\hat{\mu}_T^\star := \arg\min_{\m\geq 0} \Pih(\mu), \qquad \Pih^\star := \Pih(\hat{\mu}_T^\star)
\end{equation}
in addition to the definitions in \eqref{eq:muStar} and \eqref{eq:piStar}.
This section is devoted to the analysis of the minimizers $\muT$, $\muh$ and $\mui$, and the optimality gap for the two approximations. 

Throughout this section, we assume that $u_2, u_3 <\iy$ and $\E[Q(0)^2] <\iy$.

By its definition in \eqref{eq:PiInf} and Lemma \ref{lemma:workloadmoments}, we have an optimal expression for the steady-state cost function
\begin{equation}
\Pi_{\iy}(\mu) = \frac{\l u_2}{2(\mu-\l)} + \a\mu.
\end{equation}
It is easily verified that $\Pi_{\iy}$ is strictly convex in $\mu$, e.g. by observing that $\Pi_{\iy}''(\mu) > 0$ for all $\mu > \l$. Therefore $\Pi_{\iy}$ has a unique global minimizer and 
\begin{equation}
\label{eq:muInf}
\mui = \l + \sqrt{\frac{\l u_2}{2\a}}, \qquad  \Pi_{\iy}^\star = \a\l + \sqrt{2\a\l u_2}.
\end{equation}
We are interested in the relation between $\mui$ and $\muT$, and $\muh$ and $\muT$. 
Since $\Pi_{T}(\mu) = \Pi_{\iy}(\mu) + O(1/T)$ for all $\mu > \l$, we have pointwise convergence of the sequence $\Pi_{T}$, as well as $\hat{\Pi}_{T}$, to $\Pi_{\iy}$ for $T\to\iy$, we also expect $\muT \to \mui$ and $\muh\to\mui$ for $T\to\iy$.
Before proving that this convergence indeed holds, we result a result on the strict convexity of the function $\Pi_{T}$. 

\begin{lemma}
Let $\mu\geq 0$. The function $\Pi_{T}(\mu)$ is 
\begin{itemize}
\item convex in $\m$, if $Q(0)\equiv x$, $T<x/\mu$ and $\sigma=0$,
\item strictly convex in $\mu$, otherwise. 
\end{itemize}
\end{lemma}
Building upon strict convexity of both $\Pi_T(\mu)$ and $\Pi_\iy(\mu)$ for $\mu>\l$, we derive the following convergence result. The proof can be found in Appendix B. 
\begin{prop}
Let $\muT$, $\muh$ and $\mui$ be as defined in \eqref{eq:muStar} and \eqref{eq:muhat}. Then 
\begin{equation}
\muT \to \mui\, \qquad \text{\rm and } \qquad \muh \to \mui,
\end{equation}
for $T\to\infty$.
\end{prop}
The next result describes a refinement of $\muT$ in terms of $\mui$.
\begin{prop}\label{prop:muBullet}
For $T$ sufficiently large,
\begin{equation}
\muT = \mui + \frac{\mu_\bullet}{T} + o(1/T),
\end{equation}
where
\begin{equation}\label{eq:muBullet}
\mu_\bullet = \frac{\E[Q(0)^2]}{\sqrt{8\l u_2\a}} - \frac{u_3}{3 u_2} - 3\sqrt{\frac{\a\l u_2}{8}}.
\end{equation}
\end{prop}
Based on Proposition \ref{prop:muBullet} we propose a \emph{corrected staffing rule}, accounting for the finite horizon
\begin{equation}
\label{eq:correctedMu}
\tilde{\mu}_T^\star = \left[\mui + \frac{\mu_\bullet}{T}\right]^+, 
\end{equation}
with $\mu_\bullet$ as in \eqref{eq:muBullet}.
Here $[x]^+ := \max\{x,0\}$, which ensures the value of $\tilde{\mu}_T^\star$ is non-negative and thus is a feasible solution of the optimization problem.
This refined capacity allocation rule is expected to reduce the costs incurred in transient settings.
 However, the value we are particularly interested in is the cost increase for using either one of the approximations rather than the actual minimum $\muT$, that is, the \emph{optimality gap}.
 As it happens, we deduce the order of the optimality gap for $\mui$ with the help of the explicit form of $\mu_\bullet$ given in \eqref{eq:correctedMu}, which is stated in the next proposition. The proof is given in Appendix \ref{sec:proofProp4}

\begin{prop}\label{prop:optimalitygap_mui}
Let $\mui$ be as in \eqref{eq:muInf}. Then,
\begin{equation}
\Pi_T(\mui) - \Pi_T^\star = O(1/T^2).
\end{equation}
\end{prop}

\section{Numerical experiments}
\label{sec:numerics}

\subsection{Influence of $\Omega_{T}(\mu)$}

We first assess the contribution of the correction to the cost function provided by Theorem 1. In other words, we investigate whether $\hat{\Pi}_{T}(\mu)$ as in \eqref{eq:PiT} yields a significantly better fit to $\Pi_{T}(\mu)$, than $\Pi_{\iy}(\mu)$ does. 
Note that these three functions only differ in the costs describing the congestion. 
Therefore, we limit our study in this subsection to the evaluation of $C_T(\mu)$ as in \eqref{eq:CTmu} with stationary equivalent $C_{\iy}(\mu) = \E[Q_{\m}(\iy)]$. 
Our novel approximation hence reads 
\begin{equation}
\hat{C}_{T}(\mu) := C_{\infty}(\mu) + \Omega_{T}(\mu),
\end{equation}
with $\Omega_{T}(\mu)$ given in \eqref{eq:mainResult}.

We conduct our numerical experiments based on three models, namely: 

\begin{enumerate}
\item \underline{$M/M/1$ queue}: $U(t)$ is a unit rate compound Poisson process with exponentially distributed increments. We have $u_2 = 2$, $u_3=3$, so that
\begin{equation}\label{eq:MM1cor}
\hat{C}_{T}(\mu) = \frac{\l}{\mu-\l} + \frac{1}{T(\mu-\l)} \left(\frac{x^2}{2} - \frac{\l^2}{(\mu-\l)^2} - \frac{\l}{\mu-\l}\right).
\end{equation}
\item \underline{$M/{\rm Pareto}/1$ queue}: $U(t)$ is a unit rate compound Poisson process with Pareto increments. The Pareto distribution deserves special attention due to its heavy-tailed nature, having tail probability $\bar{F}(x) = (x/k)^{-\g}$, if $x\geq k$ and 1 otherwise. 
It is well-known that heavy-tailed service times lead to long relaxation time. For our purposes, we fix shape parameter $\g = 16/5$ and scale parameter $k=11/16$, so that $\b = 1$, $u_2 = 121/96$, $u_3 = 1331/256$ and $u_k=\iy$ for all $k>3$. Hence, 
\begin{equation}
\label{eq:MP1cor}
\hat{C}_{T}(\mu) = \frac{121\l}{192(\mu-\l)} + \frac{1}{2T(\mu-\l)}
\left( x^2 - \frac{(121\l/96)^2}{2(\mu-\l)^2} - \frac{ 1331\l/256 }{2(\mu-\l)}\right)
\end{equation}
\item \underline{Reflected Brownian motion}: $U(t)$ is Brownian motion with drift 1 and infinitesimal variance $\s^2$. We have $u_2 = \sigma^2$, $u_3=0$, so that
\begin{equation}\label{eq:RBMcor}
\hat{C}_{T}(\mu) = \frac{\l\sigma^2}{2(\mu-\l)} + \frac{1}{2T(\mu-\l)} \left( x^2 - \frac{\l^2\sigma^4}{2(\mu-\l)^2}\right).
\end{equation}
\end{enumerate} 

Let $C_T^{\l}(\mu)$ denote the cost function given arrival rate $\l$.
Although we want to explore a variety of parameter settings for these three settings, one can deduce from the identity in \eqref{eq:Qidentity} that $C_{T}^{\l}(\mu) \equiv C_{\l T}^{1}(\mu/\l)$ and $\Omega_{T}^\l(\mu) \equiv \Omega_{\l T}^1(\mu/\l)$.
This implies that it suffices to evaluate the systems for $\l=1$, since we directly obtain the measures for any other value of $\l$ by scaling the variable $\mu$ and parameter $T$ appropriately.  

For the $M/M/1$ and $M/{\rm Pareto}/1$ queue, we obtained the function $C_{T}(\mu)$ with $\l=1$ through simulation and are accurate up until a 95\% confidence interval of width $10^{-3}$. For reflected Brownian motion, we used the explicit distribution function given in \cite{Harrison1985} for double numerical integration. The results for several values of $T$ and two different starting states are depicted in Figures 4-6. These plots also include the approximated functions $\hat{C}_{T}(\mu)$. 
\begin{figure}[h]
\centering
\begin{subfigure}{0.48\textwidth}
\includegraphics[scale=0.55]{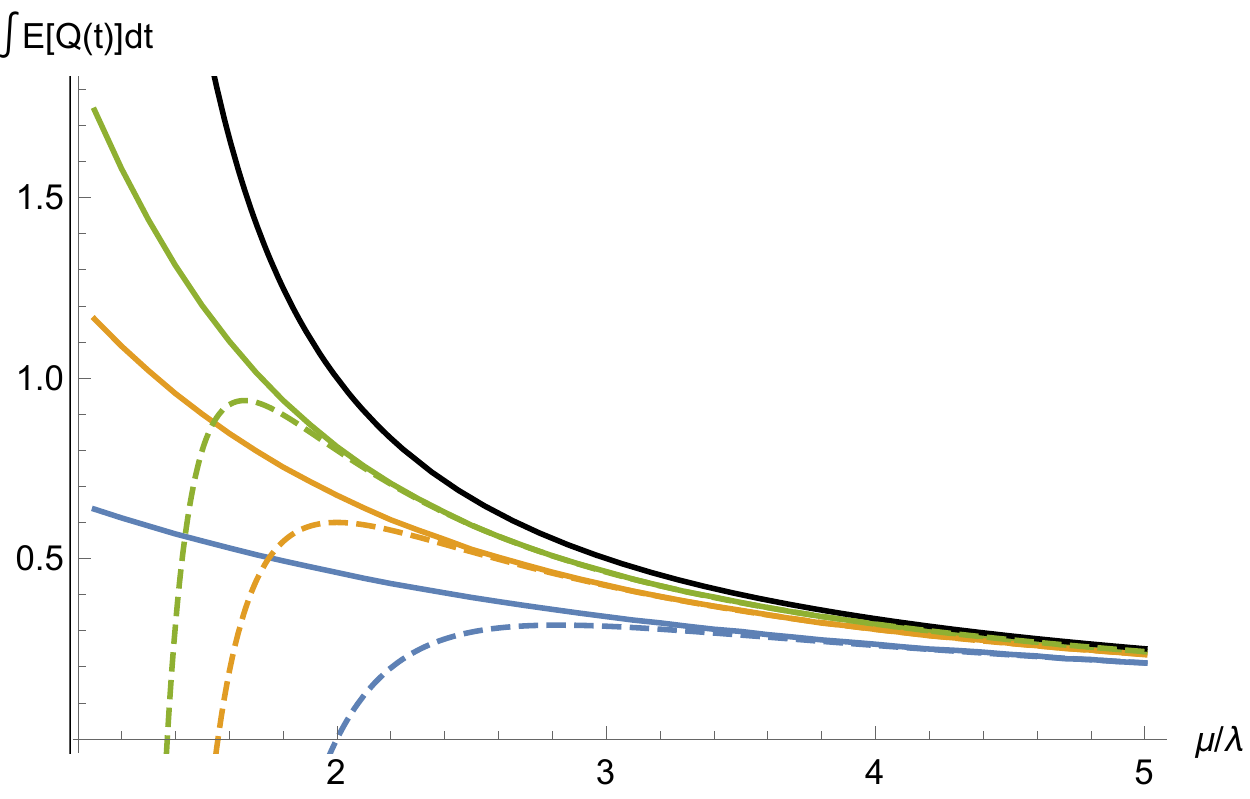}
\caption{$x=0$}
\end{subfigure}
\begin{subfigure}{0.48\textwidth}
\includegraphics[scale=0.55]{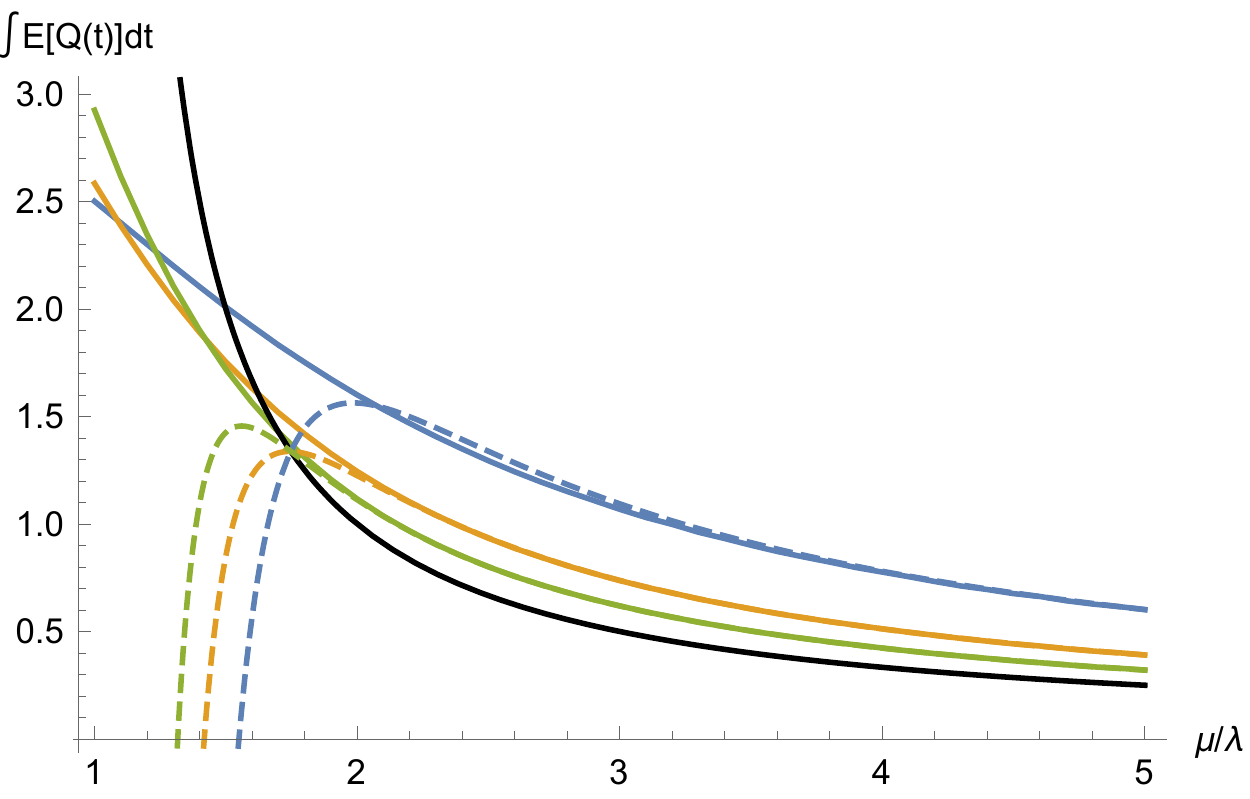}
\caption{$x=2.5$}
\end{subfigure}
\caption{$C_{T}(\mu)$ as a function of $\mu$ for $M/M/1$ for $T=2$ (blue), $T=5$ (yellow) and $ T=10$ (green) with their approximative equivalents $\hat{C}_{T}(\mu)$ (dashed) and $C_{\iy}(\mu)$ (black).}
\end{figure}
\begin{figure}[h!]
\centering
\begin{subfigure}{0.48\textwidth}
\includegraphics[scale=0.55]{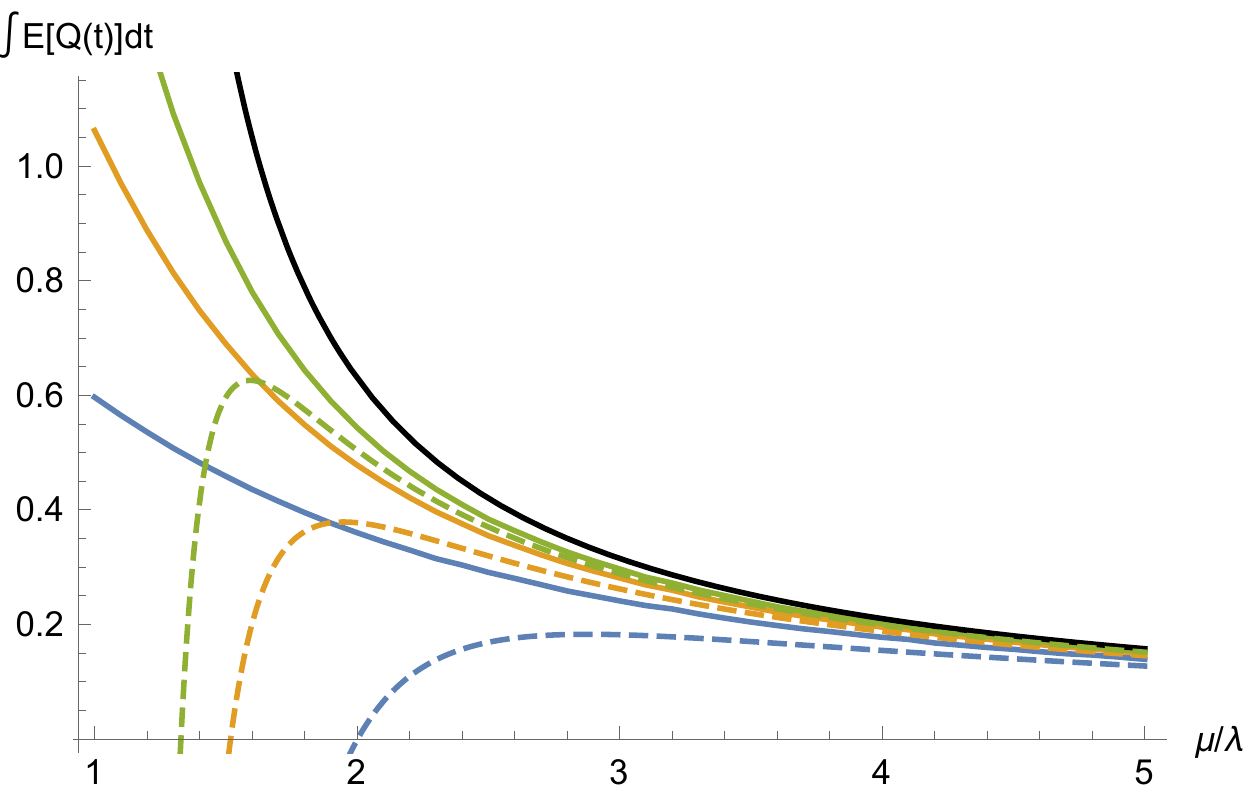}
\caption{$x=0$}
\end{subfigure}
\begin{subfigure}{0.48\textwidth}
\includegraphics[scale=0.55]{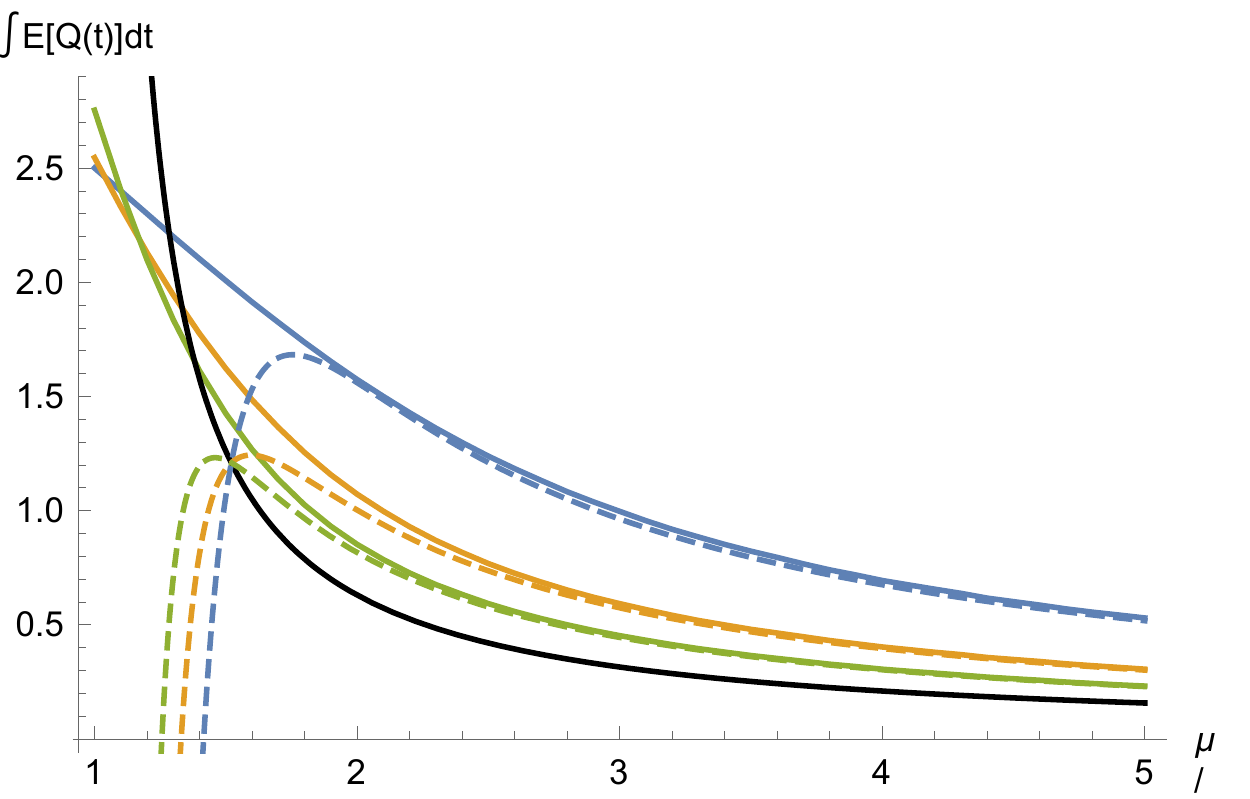}
\caption{$x=2.5$}
\end{subfigure}
\caption{$C_{T}(\mu)$ as a function of $\mu$ for $M/{\rm Pareto}/1$ for $T=2$ (blue), $T=5$ (yellow) and $T=10$ (green) with their approximative equivalents $\hat{C}_{T}(\mu)$ (dashed) and $C_{\iy}(\mu)$ (black).}
\end{figure}
\begin{figure}[h!]
\centering
\begin{subfigure}{0.48\textwidth}
\includegraphics[scale=0.55]{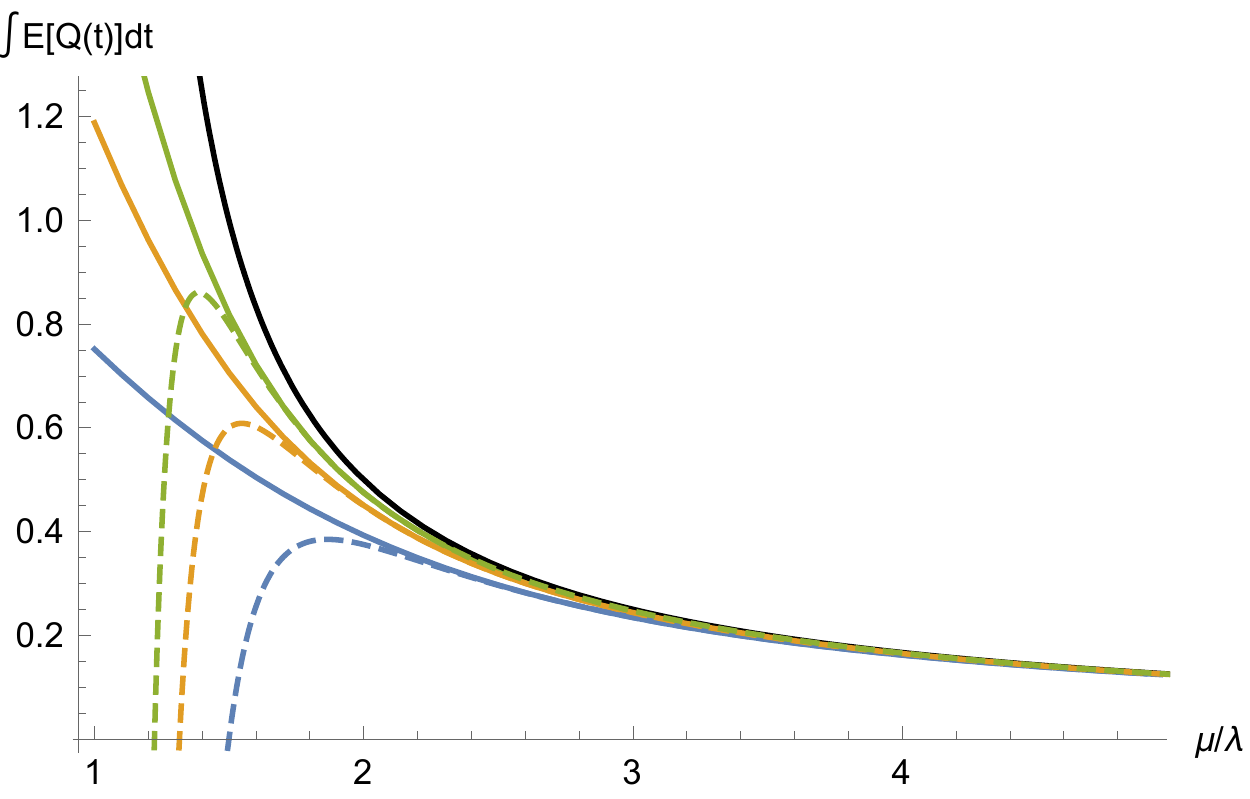}
\caption{$x=0$}
\end{subfigure}
\begin{subfigure}{0.48\textwidth}
\includegraphics[scale=0.55]{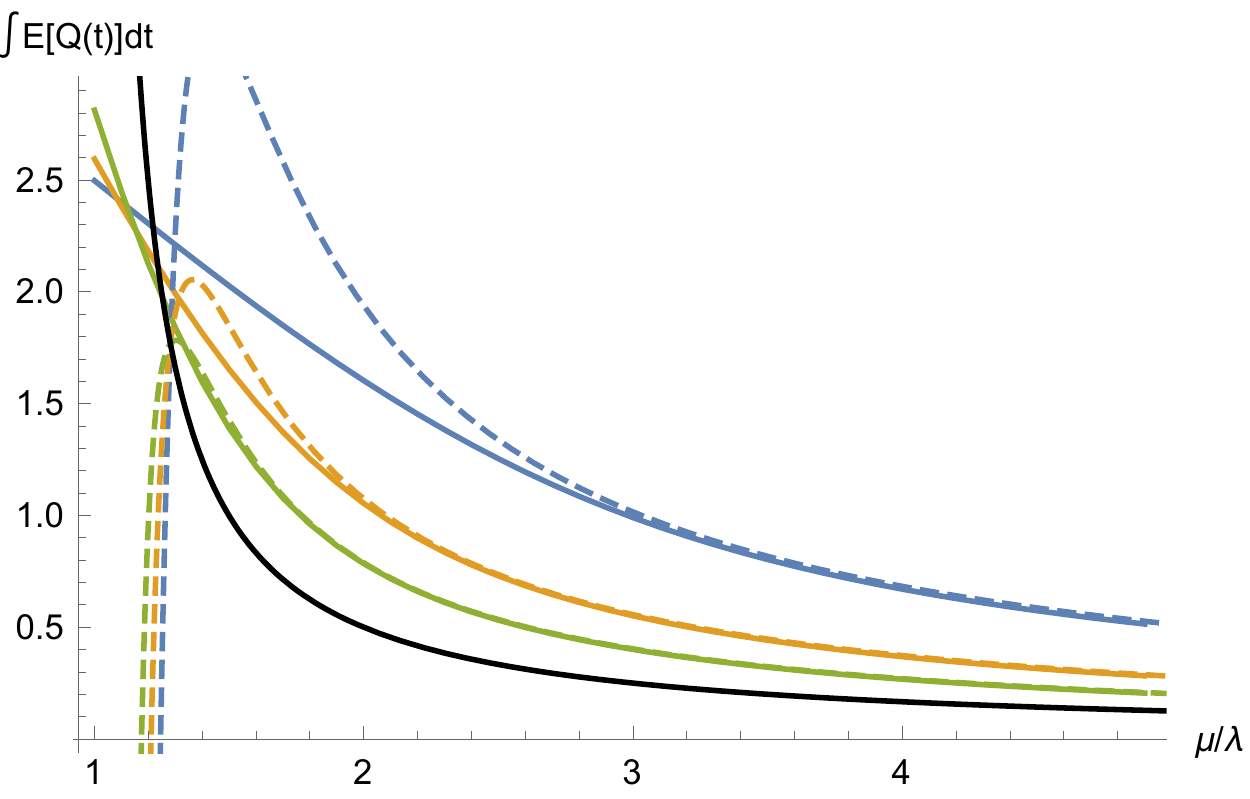}
\caption{$x=2.5$}
\end{subfigure}
\caption{$C_{T}(\mu)$ as a function of $\mu$ for RBM with $\sigma=1$ for $T=2$ (blue), $T=5$ (yellow) and $T=10$ (green) with their approximative equivalents $\hat{C}_{T}(\mu)$ (dashed) and $C_{\iy}(\mu)$ (black).}
\end{figure}
We name a few observations based on these figures. 

First, we indeed note the pointwise convergence of $\hat{C}_{T}(\mu)$ to $\hat{C}_{\iy}(\mu)$ as $T$ grows, for all $\mu$ in all three cases. However, the difference between the stationary costs and those for small values of $T$ can be significant. This is most clear in the plots with $x=2.5$ and when $\mu$ is close to $\l$, i.e. it is in heavy-traffic. In these scenarios, it is evident that refinements of the stationary costs are needed. $\hat{C}_{T}(\mu)$ does a fairly good job at providing such correction, especially for moderate values of $\mu$. 

Furthermore, we note that $C_{T}(\mu)$ approaches $C_{\iy}(\mu)$ from below for $x=0$ for any value of $\mu$, while this is not strictly the case for $x>0$. 
$\hat{C}_{T}(\mu)$ correctly captures the sign of this correction. 

Finally, observe that $\hat{C}_{T}(\mu)\to -\iy$ as $\mu$ approaches $\l$. This divergence is clear from the expressions in \eqref{eq:MM1cor}-\eqref{eq:RBMcor}. 
Our correction term relies on the premise that under the coupling scheme, the sample paths of the two queues starting from different states have hit with high probability.
This is equivalent to saying that the `largest' of the two queues is has emptied at least once before time $T$. However, as $\mu$ approaches $\l$, the system enters heavy traffic, and hence the hitting time of the zero barrier is set to run off to infinity. 
Consequently, this causes our approximation to be inaccurate for small values of $\mu$.

\subsection{Validation of corrected staffing rule}

In this section, we examine whether the corrected staffing rule $\tilde{\mu}_T^\star$ as in \eqref{eq:correctedMu} indeed yields a significant cost reduction over the choice of $\mui$ by comparing their true costs $\Pi_{T}(\tilde{\mu}_T^\star)$ and $\Pi_{T}(\mui)$.
We conduct this comparison for different values of the parameters, $\a$, $T$ and starting state $x$ through numerical experiments.  
The three models on which we do our calculations are the $M/M/1$ queue, the $M/$Pareto$/1$ queue and the reflected Brownian motion, as introduced in the previous subsection. 
Extending the reasoning of the previous subsection, saying that the cost function for general $\l>0$ can be fit into the case $\l=1$ while applying additional scaling, we focus on the latter scenario only. 

For each of the three models, we adhere to the following set-up. The quality of both staffing rules is assessed for $\a = 0.1, 1$ and 2, resembling three modes of valuation of the QoS in the system. 
As a benchmark, observe that the expected workload in steady-state conditions with staffing level $\mui$ equals
\begin{equation}
C_\iy(\mui) = \sqrt{\frac{\a\l u_2}{2}}.
\end{equation}
For each value of $\a$, we consider two scenarios: one in which the system starts empty, i.e. $x=0$, and one in which the initial state is double this benchmark value, thus $x=\sqrt{2\a\l u_2}$.
The numerics will be presented for each model separately. Afterwards, we discuss the conclusions we are able to draw from these results.

\subsubsection*{$M/M/1$ queue}
As we discussed before, if $U$ is a unit rate compound Poisson process with exponentially distributed increments, then $\Qlm$ describes the workload process in an $M/M/1$ queue. 
For this setting we get

\begin{equation}
 \mui = \l + \sqrt{\frac{\l}{\a}},\qquad \tilde{\mu}_T^\star = \left[\l + \sqrt{\frac{\l}{\a}} + \frac{1}{T}\left( \frac{x^2}{4\sqrt{\l\a}} - 1 - \frac{3}{2} \sqrt{\l\a}\right)\right]^+.
 \end{equation} 
 
 Table \ref{tab:mm1} presents the actual costs corresponding to these two staffing levels for different value of $x$ and $\a$. 
 
\begin{table}[h!]
\centering\resizebox{15cm}{!} {
\begin{tabular}{|c|r|rr|rr|r||rr|rr|r|}
\cline{3-12}
\multicolumn{1}{c}{} & \multicolumn{1}{r|}{} & \multicolumn{5}{c||}{$x = 0$} & \multicolumn{5}{c|}{$x = 2\sqrt{\a}$} \\
\hline
$\a$ & $T$     & $\mui$  & $\Pi_T(\mui)$ & $\tilde{\mu}_T^\star$  & $\Pi_T(\tilde{\mu}_T^\star)$ & \% & $\mui$  & $\Pi_T(\mui)$ & $\tilde{\mu}_T^\star$  & $\Pi_T(\tilde{\mu}_T^\star)$ & \% \\
\hline
\multicolumn{1}{|c|}{\multirow{4}[2]{*}{0.1}} & 1     & 4.162 & 0.620 & 2.688 & 0.536 & 0.136 & 4.162 & 0.682 & 2.688 & 0.536 & 0.214 \\
\multicolumn{1}{|c|}{} & 2     & 4.162 & 0.669 & 3.425 & 0.641 & 0.041 & 4.162 & 0.700 & 3.425 & 0.641 & 0.085 \\
\multicolumn{1}{|c|}{} & 5     & 4.162 & 0.706 & 3.867 & 0.703 & 0.005 & 4.162 & 0.719 & 3.867 & 0.703 & 0.022 \\
\multicolumn{1}{|c|}{} & 10    & 4.162 & 0.719 & 4.015 & 0.719 & 0.001 & 4.162 & 0.726 & 4.015 & 0.719 & 0.010 \\
\hline
\hline
\multicolumn{1}{|c|}{\multirow{4}[2]{*}{1}} & 1     & 2.000 & 2.309 & 0.000 & 0.500 & 0.783 & 2.000 & 3.500 & 0.500 & 2.750 & 0.214 \\
\multicolumn{1}{|c|}{} & 2     & 2.000 & 2.461 & 0.750 & 1.480 & 0.398 & 2.000 & 3.218 & 1.250 & 3.125 & 0.029 \\
\multicolumn{1}{|c|}{} & 5     & 2.000 & 2.675 & 1.500 & 2.400 & 0.103 & 2.000 & 3.043 & 1.700 & 2.968 & 0.025 \\
\multicolumn{1}{|c|}{} & 10    & 2.000 & 2.810 & 1.750 & 2.726 & 0.030 & 2.000 & 3.007 & 1.850 & 2.980 & 0.009 \\
\hline
\hline
\multicolumn{1}{|c|}{\multirow{4}[2]{*}{2}} &1     & 1.707 & 3.744 & 0.000 & 0.500 & 0.866 & 1.707 & 5.889 & 0.000 & 3.328 & 0.435 \\
\multicolumn{1}{|c|}{} &2     & 1.707 & 3.924 & 0.146 & 1.232 & 0.686 & 1.707 & 5.547 & 0.854 & 4.682 & 0.156 \\
\multicolumn{1}{|c|}{} &5     & 1.707 & 4.209 & 1.083 & 3.343 & 0.206 & 1.707 & 5.114 & 1.366 & 4.910 & 0.040 \\
\multicolumn{1}{|c|}{} &10    & 1.707 & 4.424 & 1.395 & 4.108 & 0.071 & 1.707 & 4.945 & 1.536 & 4.868 & 0.016 \\
\hline
\end{tabular}}
\caption{Comparison of costs for the $M/M/1$ queue for steady-state and corrected staffing rules.}
\label{tab:mm1}
\end{table}
%


\subsubsection*{M/{\rm Pareto}/1 queue}

In case the service requirements follow a Pareto distribution with shape parameter $\gamma = 16/5$, the staffing rules become
\begin{equation}
 \mui = \l + \frac{11}{8}\sqrt{\frac{ \l }{3 \a}},\qquad \tilde{\mu}_T^\star = \left[\l + \frac{11}{8}\sqrt{\frac{ \l }{3 \a}} + \frac{1}{T}\left( \frac{2 x^2}{11\sqrt{\l\a/3}} - \frac{11}{8} - \frac{11}{16} \sqrt{3\l\a}\right)\right]^+.
 \end{equation} 
The numerical results are given in Table \ref{tab:mp1}. 
\begin{table}[h!]
\centering\resizebox{15cm}{!} {
\begin{tabular}{|c|r|rr|rr|r||rr|rr|r|}
\cline{3-12}
\multicolumn{1}{c}{} & \multicolumn{1}{r|}{} & \multicolumn{5}{c||}{$x = 0$} & \multicolumn{5}{c|}{$x = 11/4\cdot \sqrt{\a/3}$} \\
\hline
$\a$ & $T$     & $\mui$  & $\Pi_T(\mui)$ & $\tilde{\mu}_T^\star$  & $\Pi_T(\tilde{\mu}_T^\star)$ & \% & $\mui$  & $\Pi_T(\mui)$ & $\tilde{\mu}_T^\star$  & $\Pi_T(\tilde{\mu}_T^\star)$ & \% \\
\hline
\multicolumn{1}{|c|}{\multirow{4}[2]{*}{0.1}} & 1     & 3.510 & 0.524 & 1.759 & 0.461 & 0.120 & 3.510 & 0.573 & 2.010 & 0.562 & 0.019 \\
\multicolumn{1}{|c|}{} & 2     & 3.510 & 0.555 & 2.635 & 0.539 & 0.029 & 3.510 & 0.580 & 2.760 & 0.574 & 0.010 \\
\multicolumn{1}{|c|}{} & 5     & 3.510 & 0.580 & 3.160 & 0.578 & 0.003 & 3.510 & 0.591 & 3.210 & 0.589 & 0.002 \\
\multicolumn{1}{|c|}{} & 10    & 3.510 & 0.590 & 3.335 & 0.590 & 0.000 & 3.510 & 0.596 & 3.360 & 0.595 & 0.001 \\
\hline
\hline
\multicolumn{1}{|c|}{\multirow{4}[2]{*}{1}} & 1     & 1.794 & 2.076 & 0.000 & 0.500 & 0.759 & 1.794 & 2.989 & 0.000 & 2.088 & 0.302 \\
\multicolumn{1}{|c|}{} & 2     & 1.794 & 2.190 & 0.511 & 1.291 & 0.411 & 1.794 & 2.790 & 0.610 &  2.588 & 0.072 \\
\multicolumn{1}{|c|}{} & 5     & 1.794 & 2.345 & 1.281 & 2.108 & 0.101 & 1.794 & 2.638 & 1.320 & 2.607 & 0.012 \\
\multicolumn{1}{|c|}{} & 10    & 1.794 & 2.441 & 1.537 & 2.371 & 0.029 & 1.794 & 2.597 & 1.557 & 2.585 & 0.005 \\
\hline
\hline
\multicolumn{1}{|c|}{\multirow{4}[2]{*}{2}} & 1     & 1.561 & 3.427 & 0.000 & 0.500 & 0.854 & 1.561 & 5.087 & 0.000 & 2.745 & 0.460 \\
\multicolumn{1}{|c|}{} & 2     & 1.561 & 3.567 & 0.032 & 1.050 & 0.706 & 1.561 & 4.832 & 0.172 & 3.417 & 0.293 \\
\multicolumn{1}{|c|}{} & 5     & 1.561 & 3.779 & 0.950 & 3.012 & 0.203 & 1.561 & 4.499 & 1.006 & 4.313 & 0.041 \\
\multicolumn{1}{|c|}{} & 10    & 1.561 & 3.935 & 1.255 & 3.356 & 0.147 & 1.561 & 4.351 & 1.284 & 4.304 & 0.011 \\
\hline
\end{tabular}
}
\caption{Comparison of costs for the $M/{\rm Pareto}/1$ queue for steady-state and corrected staffing rules.}
\label{tab:mp1}
\end{table}
Just as in the results for the $M/M/1$ queue, we observe a higher reduction for larger value of $\a$ and $T$. Also, again $\tilde{\mu}_T < \mui$. Hence, the conclusions for the $M/{\rm Pareto}/1$ queue are similar to those of the $M/M/1$ queue. 

\subsubsection*{Reflected Brownian motion}

In case the input process $U$ is Brownian motion with drift 1 and infinitesimal variance $\s^2$, the steady-state staffing rule and its corrected version reduce to 

\begin{equation}
\mui = \l + \sqrt{\frac{\l\s^2}{2\a}}, \qquad  
\tilde{\mu}_T^\star = \left[\l + \sqrt{\frac{\l\s^2}{2\a}} + \frac{1}{2\sqrt{2}\,T}\left(\frac{x^2}{\sqrt{\l \a}\s} - 3\s\sqrt{\a\l} \right)\right]^+.
\end{equation}

In Tables \ref{tab:rbm1} and \ref{tab:rbm2}, the costs obtained through numerical evaluation are presented for several values of $x$, $T$. We also vary $\s$ to examine the influence of the volatility of arrival process on the quality of the staffing rules.

The observations on the influence of $\a, x$ and $T$ are similar to those of the $M/M/1$ queue and the $M/{\rm Pareto}/1$ queue. 
However, here we see little improvement induced by the corrected staffing rule for small values of $\a$ for both values of $x$. 
The results in Tables \ref{tab:rbm1}-\ref{tab:rbm2} also suggest that the reduction is smaller for larger values of $\s$. 

\begin{table}
\centering
\resizebox{15cm}{!} {
\begin{tabular}{|c|r|rr|rr|r||rr|rr|r|}
\cline{3-12}
\multicolumn{1}{c}{} & \multicolumn{1}{r|}{} & \multicolumn{5}{c||}{$x = 0$} & \multicolumn{5}{c|}{$x = \sqrt{2\a} $} \\
\hline
$\a$ & $T$     & $\mui$  & $\Pi_T(\mui)$ & $\tilde{\mu}_T^\star$  & $\Pi_T(\tilde{\mu}_T^\star)$ & \% & $\mui$  & $\Pi_T(\mui)$ & $\tilde{\mu}_T^\star$  & $\Pi_T(\tilde{\mu}_T^\star)$ & \% \\
\hline
\multicolumn{1}{|c|}{\multirow{4}[2]{*}{0.1}} & 1     & 3.236 & 0.525 & 2.901 & 0.518 & 0.013 & 3.236 & 0.565 & 3.124 & 0.564 & 0.001 \\
\multicolumn{1}{|c|}{} & 2     & 3.236 & 0.536 & 3.068 & 0.534 & 0.003 & 3.236 & 0.556 & 3.180 & 0.556 & 0.000 \\
\multicolumn{1}{|c|}{} & 5     & 3.236 & 0.543 & 3.169 & 0.542 & 0.000 & 3.236 & 0.551 & 3.214 & 0.551 & 0.000 \\
\multicolumn{1}{|c|}{} & 10    & 3.236 & 0.545 & 3.203 & 0.545 & 0.000 & 3.236 & 0.549 & 3.225 & 0.549 & 0.000 \\
\hline
\hline
\multicolumn{1}{|c|}{\multirow{4}[2]{*}{1}} & 1     & 1.500 & 3.420 & 0.000 & 0.833 & 0.756 & 1.500 & 4.741 & 1.000 & 3.984 & 0.160 \\
\multicolumn{1}{|c|}{} & 2     & 1.500 & 3.539 & 0.750 & 2.386 & 0.326 & 1.500 & 4.579 & 1.250 & 4.293 & 0.063 \\
\multicolumn{1}{|c|}{} & 5     & 1.500 & 3.707 & 1.200 & 3.363 & 0.093 & 1.500 & 4.335 & 1.400 & 4.274 & 0.014 \\
\multicolumn{1}{|c|}{} & 10    & 1.500 & 3.820 & 1.350 & 3.705 & 0.030 & 1.500 & 4.190 & 1.450 & 4.175 & 0.004 \\
\hline
\hline
\multicolumn{1}{|c|}{\multirow{4}[2]{*}{2}} & 1     & 1.500 & 3.420 & 0.000 & 0.833 & 0.756 & 1.500 & 4.741 & 1.000 & 3.984 & 0.160 \\
\multicolumn{1}{|c|}{} & 2     & 1.500 & 3.539 & 0.750 & 2.386 & 0.326 & 1.500 & 4.579 & 1.250 & 4.293 & 0.063 \\
\multicolumn{1}{|c|}{} & 5     & 1.500 & 3.707 & 1.200 & 3.363 & 0.093 & 1.500 & 4.335 & 1.400 & 4.274 & 0.014 \\
\multicolumn{1}{|c|}{} & 10    & 1.500 & 3.820 & 1.350 & 3.705 & 0.030 & 1.500 & 4.190 & 1.450 & 4.175 & 0.004 \\
\hline
\end{tabular}}
\caption{Comparison of costs for RBM with $\sigma = 1$ for steady-state and corrected staffing rules}
\label{tab:rbm1}
\end{table}


\begin{table}
\centering
\resizebox{15cm}{!} {
\begin{tabular}{|c|r|rr|rr|r||rr|rr|r|}
\cline{3-12}
\multicolumn{1}{c}{} & \multicolumn{1}{r|}{} & \multicolumn{5}{c||}{$x = 0$} & \multicolumn{5}{c|}{$x = 2\sqrt{2\a} $} \\
\hline
$\a$ & $T$     & $\mui$  & $\Pi_T(\mui)$ & $\tilde{\mu}_T^\star$  & $\Pi_T(\tilde{\mu}_T^\star)$ & \% & $\mui$  & $\Pi_T(\mui)$ & $\tilde{\mu}_T^\star$  & $\Pi_T(\tilde{\mu}_T^\star)$ & \% \\
\hline
\multicolumn{1}{|c|}{\multirow{4}[2]{*}{0.1}} & 1     & 5.472 & 0.950 & 4.801 & 0.936 & 0.015 & 5.472 & 1.030 & 5.249 & 1.029 & 0.001 \\
\multicolumn{1}{|c|}{} & 2     & 5.472 & 0.972 & 5.137 & 0.968 & 0.003 & 5.472 & 1.012 & 5.360 & 1.012 & 0.000 \\
\multicolumn{1}{|c|}{} & 5     & 5.472 & 0.985 & 5.338 & 0.985 & 0.000 & 5.472 & 1.002 & 5.427 & 1.002 & 0.000 \\
\multicolumn{1}{|c|}{} & 10    & 5.472 & 0.990 & 5.405 & 0.990 & 0.000 & 5.472 & 0.998 & 5.450 & 0.998 & 0.000 \\
\hline
\hline
\multicolumn{1}{|c|}{\multirow{4}[2]{*}{1}} & 1     & 2.414 & 3.176 & 0.293 & 1.546 & 0.513 & 2.414 & 4.633 & 1.707 & 4.228 & 0.087 \\
\multicolumn{1}{|c|}{} & 2     & 2.414 & 3.356 & 1.354 & 2.690 & 0.199 & 2.414 & 4.375 & 2.061 & 4.247 & 0.029 \\
\multicolumn{1}{|c|}{} & 5     & 2.414 & 3.573 & 1.990 & 3.411 & 0.045 & 2.414 & 4.094 & 2.273 & 4.073 & 0.005 \\
\multicolumn{1}{|c|}{} & 10    & 2.414 & 3.689 & 2.202 & 3.646 & 0.012 & 2.414 & 3.966 & 2.344 & 3.962 & 0.001 \\
\hline
\hline
\multicolumn{1}{|c|}{\multirow{4}[2]{*}{2}} & 1     & 2.000 & 4.839 & 0.000 & 1.339 & 0.723 & 2.000 & 7.481 & 1.000 & 5.967 & 0.202 \\
\multicolumn{1}{|c|}{} & 2     & 2.000 & 5.078 & 0.500 & 2.773 & 0.454 & 2.000 & 7.158 & 1.500 & 6.585 & 0.080 \\
\multicolumn{1}{|c|}{} & 5     & 2.000 & 5.414 & 1.400 & 4.726 & 0.127 & 2.000 & 6.670 & 1.800 & 6.549 & 0.018 \\
\multicolumn{1}{|c|}{} & 10    & 2.000 & 5.639 & 1.700 & 5.409 & 0.041 & 2.000 & 6.380 & 1.900 & 6.349 & 0.005 \\
\hline
\end{tabular}}
\caption{Comparison of costs for RBM with $\sigma = 2$ for steady-state and corrected staffing rules.}
\label{tab:rbm2}
\end{table}

\subsection{Discussion}

Based upon these numerical results in Tables \ref{tab:mm1}-\ref{tab:rbm2}, we make a few remarks. The three models roughly exhibit similar behavior as $T$, $x$ and $\a$ are varied.

Non-surprisingly, we note that $\tilde{\mu}_T$ approaches $\mui$ with increasing $T$, which also implies that the cost reduction achieved by the corrected staffing rule vanishes as $T\to\iy$. 
Also, we observe that in all scenarios examined, the cost reduction increases with $\a$. This can be explained through investigation of the objective function $\Pi_T$ as function of $\m$. Namely, for $\a$ small, the curve is relatively flat around the true optimum $\muT$. Hence, in this case a moderate deviation from $\muT$ will likely not lead to a significant cost increase. However, as $\a$ becomes larger, i.e. the server efficiency is valued more than minimization of congestion, the curve becomes more sharp around $\muT$, and hence more accurate approximations of $\muT$ are required to achieve an acceptable cost level. Hence, the corrected staffing rule \eqref{eq:correctedMu} proves particularly useful in these cases.

Another point we want to highlight is that the relative improvement is higher for $x=0$, as opposed to $x=\sqrt{2\a\l u_2}$. Moreover, even though the initial state of the system is above the optimal equilibrium, $\tilde{\mu}_T$ is smaller than $\mui$. This is somewhat counter-intuitive. In fact, from \eqref{eq:muBullet} it follows that $\mu_\bullet$ positively contributes to the corrected staffing function if
\begin{equation}
\E[Q^2(0)] >  3\a\l u_2 + \frac{2 u_2}{3 u_3}\,\sqrt{2\a\l u_2}.
\end{equation}
Even more surprisingly, obverse that if $Q(0)\equiv Q(\iy)$ with $\mu \equiv \mui = \l + \sqrt{\l u_2/(2\a)}$ we get with Lemma \ref{lemma:workloadmoments}
\begin{equation}
\E[Q^2(0)] = \a\l u_2 + \frac{u_3}{3} \sqrt{\frac{2\a\l}{u_2}},
\end{equation}
so that
\begin{equation}
\mu_\bullet = {-}\frac{u_3}{2 u_2} - \sqrt{2\a\l u_2} < 0.
\end{equation}
This suggests that even when the process is started in equilibrium with the corresponding optimal steady-state speed $\mui$, it is more cost efficient to change the server speed. This seems strange, but we provide an explanation for this phenomenon. 
In out particular setting, we strictly focus on the period $[0,T]$, and do not care about what happens after time $T$. Hence, it might be beneficial to let the queue build up towards the end of the period, thereby employing a smaller server speed than stipulated by the steady-state optimum. 
Naturally, this effect diminishes with $T$.
\section{Conclusion \& further research}

Motivated by the time-varying nature of queues in practical applications, we studied the impact that the transient phase has on traditional capacity allocation questions.
By defining a cost minimization problem, in which the objective function contains a correction accounting for the transient period, we identified the leading and second-order behavior of the cost function as a function of the interval length $T$. 
As a by-product, this result provides an approximation for the actual cost function, which is a refinement to its stationary counterpart.
Our numerical experiments in Section 5.1 demonstrate the improved accuracy achieved by this approximation in a number of settings. 
By perturbation analysis of the optimization problem, this furthermore gives rise to a correction to the steady-state optimal capacity allocation of order $1/T$. 
The necessity of the refined capacity allocation level is substantiated by the numerics in Section 5.2, which show the cost reduction that can be achieved in the number of settings, compared to settings in which stationary metrics are used.  
Especially for small values of $T$ and large values of $\alpha$ this reduction is significant.
Additionally, these results also indicate that it is relatively safe to use the stationary cost when $T$ is moderate, or $\alpha$ is small. 
The latter reflects the scenario in which QoS to clients is much more valued than service efficiency. 
This observation links to the flat nature of the cost function around its optimal value for $\alpha$ small, a statement on the optimality gap that we formally proved in Proposition 4.

Besides the validation of our theoretical results of Sections 3 and 4, the numerical results also reveal some phenomena that require more investigation. 

As noted, our corrected capacity allocation level $\tilde{\mu}_T^\star$ is in most studied cases less than the steady-state optimal value $\mu_{\iy}^\star$. This implies that congestion levels tends to be higher under our staffing scheme then under stationary staffing. 
A possible explanation for this may be the fact that the planning period under consideration is finite. 
Clearly, in the setting we analyzed, anything that happens after time $T$ is neglected. 
Therefore, it might be beneficial from the cost perspective to end the period with a higher expected congestion level, as it does not need to be canceled out in the future. 
Related to this observation, it would be interesting to look at the setting in which staffing decisions need to be made in consecutive periods of equal length, in which the arrival rate changes at the start of each period. 
This case requires careful consideration of the correlation among the staffing decisions within the separate periods.

Another question that arises concerns the translation of our (qualitative) findings to more general queues, in particular the $M/M/s$ queue. 
Whereas in our analysis, the central decision variable is the server speed $\mu$, the variable of interest in multi-server queues is typically the number of servers. 
It may well be that similar explicit corrections to staffing levels can be deduced to account for transience.
Since our analysis heavily relies on the comparibility of the sample paths of two single-server queues, which is due to the equal negative drift for the two processes, another approach must be taken to tackle this extension.

The analysis and findings for the single-server queue with L\'evy input presented in this paper may serve a stepping stone for investigation of these more elaborate problems.

\bibliographystyle{plain}
\bibliography{bibliography}

\begin{thebibliography}{10}

\bibitem{Abate1987a}
J.~Abate and W.~Whitt.
\newblock Transient behavior of regulated {B}rownian motion, {I}: starting at
  the origin.
\newblock {\em Advances in Applied Probability}, 19(3):560--598, 1987.

\bibitem{Abate1987b}
J.~Abate and W.~Whitt.
\newblock Transient behavior of regulated {B}rownian motion, {I}{I}: non-zero
  initial conditions.
\newblock {\em Advances in Applied Probability}, 19(3):599--631, 1987.

\bibitem{Abate1987}
J.~Abate and W.~Whitt.
\newblock Transient behavior of the {M}/{M}/1 queue: Starting at the origin.
\newblock {\em Queueing Systems: Theory and Applications}, 2(1):41--65, 1987.

\bibitem{Abate1994}
J.~Abate and W.~Whitt.
\newblock Transient behavior of the {M}/{G}/1 workload process.
\newblock {\em Operations Research}, 42(4):750--764, 1994.

\bibitem{Asmussen2003}
S.~Asmussen.
\newblock {\em Applied {P}robability and {Q}ueues (second edition)}.
\newblock Springer-Verlag, New York, 2003.

\bibitem{Benes1957}
V.E. Benes.
\newblock On queues with {P}oisson arrivals.
\newblock {\em The Annals of Mathematical Statistics}, 28(3):670--677, 1957.

\bibitem{Cohen1969}
J.W. Cohen.
\newblock {\em The {S}ingle {S}erver {Q}ueue}.
\newblock North-Holland Pub. Co., 1969.

\bibitem{Gaver1959}
D.P. Gaver.
\newblock Imbedded {M}arkov chain analysis of a waiting-line process in
  continuous time.
\newblock {\em The Annals of Mathematical Statistics}, 30(3):698--720, 1959.

\bibitem{Gaver1968}
D.P. Gaver.
\newblock Diffusion approximations and models for certain congestion problems.
\newblock {\em Journal of Applied Probability}, 5(3):607--623, 1968.

\bibitem{Green1991}
L.V. Green and P.~Kolesar.
\newblock The pointwise stationary approximation for queues with non-stationary
  arrivals.
\newblock {\em Management Science}, 37(1):84--97, 1991.

\bibitem{Harrison1985}
J.M. Harrison.
\newblock {\em Brownian Motion and Stochastic Flow Systems}.
\newblock John Wiley and Sons, 1985.

\bibitem{Janssen2015}
A.J.E.M. Janssen, J.S.H. van Leeuwaarden, and B.W.J. Mathijsen.
\newblock Novel heavy-traffic regimes for large-scale service systems.
\newblock {\em SIAM Jounral on Applied Mathematics}, 75(2):787--812, 2015.

\bibitem{Janssen2008}
A.J.E.M. Janssen, J.S.H. van Leeuwaarden, and A.P. Zwart.
\newblock Gaussian expansions and bounds for the {P}oisson distribution applied
  to the erlang {B} formula.
\newblock {\em Advances in Applied Probability}, 40(1):122--143, 2008.

\bibitem{Janssen2011}
A.J.E.M. Janssen, J.S.H. van Leeuwaarden, and A.P. Zwart.
\newblock Refining square-root safety staffing by expanding {E}rlang {C}.
\newblock {\em Operations Research}, 59(6):1512--1522, 2011.

\bibitem{Kendall1951}
D.G. Kendall.
\newblock Some problems in the theory of queues.
\newblock {\em Journal of the Royal Statistical Society}, 113(2):151--185,
  1951.

\bibitem{Kyprianou2006}
A.E. Kyprianou.
\newblock {\em Introductory Lectures on Fluctuations of L\'evy Processes with
  Applications}.
\newblock Springer, 2006.

\bibitem{Massey1998}
W.A. Massey and W.~Whitt.
\newblock Uniform acceleration expansions for {M}arkov chains with time-varying
  rates.
\newblock {\em The Annals of Applied Probability}, 1998.

\bibitem{Neuts1966}
M.F. Neuts.
\newblock The single server queue with poisson input and semi-markov service
  times.
\newblock {\em Journal of Applied Probability}, 3(1):202--230, 1966.

\bibitem{Newell1982}
G.F. Newell.
\newblock {\em Applications of Queueing Theory}.
\newblock Chapman and Hall, 1982.

\bibitem{Odoni1983}
A.R. Odoni and E.~Roth.
\newblock An empirical investigation of the transient behavior of stationary
  queueing systems.
\newblock {\em Operational Research}, 31(3):432--455, 1983.

\bibitem{Pegden1982}
C.D Pegden and M.~Rosenshine.
\newblock Some new results for the m/m/1 queue.
\newblock {\em Management Science}, 28(7):821 -- 828, 1982.

\bibitem{Prabhu1964}
N.U. Prabhu.
\newblock Time-dependent results in storage theory.
\newblock {\em Journal of Applied Probability}, 1(1):1--46, 1964.

\bibitem{Randhawa2014}
R.~Randhawa.
\newblock The optimality gap of asymptotically-derived prescriptions with
  applications to queueing systems.
\newblock 2014.

\bibitem{Sato1999}
K.-I. Sato.
\newblock {\em L\'evy Processes and Infinitely Divisible Distributions.}
\newblock Cambridge University Press, 1999.

\bibitem{Steckley2007}
S.G. Steckley and S.G. Henderson.
\newblock The error in steady-state approximations for the time-dependent
  waiting time distribution.
\newblock {\em Stochastic Models}, 23(2):307--332, 2007.

\bibitem{Takacs1955}
L.~Taka\'cs.
\newblock Investigation of waiting time problems by reduction to markov
  processes.
\newblock {\em Acta Mathematica Academiae Scientiarum Hungarica},
  6(1):101--129, 1955.

\bibitem{Takacs1962}
L.~Taka\'cs.
\newblock The time dependence of a single-server queue with poisson input and
  general service times.
\newblock {\em The Annals of Mathematical Statistics}, 33(4):1340--1348, 1962.

\bibitem{Whitt1991}
W.~Whitt.
\newblock The pointwise stationary approximation is apsymptotically correct as
  the rates increase.
\newblock {\em Management Science}, 1991.

\bibitem{Zhang2012}
B.~Zhang, J.S.H. van Leeuwaarden, and A.P. Zwart.
\newblock Staffing call centers with impatient customers: refinements to
  many-server asymptotics.
\newblock {\em Operations Research}, 60(2):461--474, 2012.

\end{thebibliography}

\appendix

\section{Proofs of Section 2}

\subsection{Proof of Lemma \ref{lemma:workloadmoments}}

\begin{proof}
The conditions of \cite[Cor.IX3.4]{Asmussen2003} are satisfied and therefore $Q_{\mu}(t)\Rightarrow \Qlm(\infty)$ in distribution for $t\rightarrow\infty$. Furthermore, its Laplace transform is for ${\rm Re}(s) < 0$
\[\tilde{Q}_{\mu}(s) = \E[s \Qlm(\infty)] = \frac{s \k_{\mu}'(0)}{\k_{\mu}(s)} = \frac{s(\l\k_U'(0) - \mu)}{\l\k_U(s) - \mu s} = \frac{s(\mu-\l)}{\mu s-\l \k_U(s)}.\]
It can be checked that $\k_U'(0) = \E[U(1)] = 1$, $\k_U''(0) = u_2$ and $\k_U'''(0) = u_3$, and $\klm'(0) = \l-\mu$, $\klm''(0) = \l u_2$ and $\klm'''(0) = \l u_3$.
Using l'H\^opital's rule we obtain the first moment of $\Qlm(\infty)$:
\begin{align*}
\E[\Qlm(\iy)] &= {-}\lim_{s\to 0} \frac{d}{ds} \tilde{Q}_{\mu}(s) = \lim_{s\to 0} \klm'(0)\, \frac{s\klm'(s)-\klm(s)}{\klm(s)^2}\\
&= \lim_{s\to 0} \klm'(0)\, \frac{{-}s\klm''(s)}{2\klm(s)\klm'(s)} 
= \lim_{s\to 0} \klm'(0)\,\frac{ s\klm'''(s)-\klm''(s)}{2\klm'(s)^2 + 2\klm(s)\klm'''(s)} \\
&= {-}\frac{\klm''(0)}{2\klm'(0)} = \frac{\l u_2}{2(\mu-\l)}.
\end{align*}
Similarly we derive the second moment:
\begin{align*}
\E[\Qlm^2(\iy)] &= \lim_{s\to 0} \frac{d^2}{ds^2} \tilde{Q}_{\mu}(s) 
= \lim_{s\to 0} \klm'(0)\, \frac{3 \klm''(0)^2 - 2\klm'(0)\klm'''(0)}{6 \klm'(0)^3}\\
&= (\l-\mu)\frac{3\l^2u_2^2 - 2\l u_3(\l-\mu)}{6(\l-\mu)^3}
= \frac{\l^2u_2^2}{2(\mu-\l)^2} +\frac{\l u_3}{3(\mu-\l)}.
\end{align*}
\end{proof}

\section{Proofs of Section 3}
\subsection{Proof of Lemma 2}
\begin{proof}
Using the representation in \eqref{eq:Yxy} we write
\begin{align*}
\Psi^{x,y}_T &= \frac{1}{T}\int_0^{\infty} \E[Y_{x,y}(t)]\dd t \\
&= \frac{1}{T}\,\E\left[\int_0^{\t^x(0)}Y_{x,y}(t)\right] \dd t + \frac{1}{T}\,\E\left[\int_0^{\t^y(0)}Y_{x,y}(t) \dd t\right] + \frac{1}{T}\,\E\left[\int_{\t^y(0)}^{\t^x(0)} Y_{x,y}(t) \dd t\right],\\
&= \frac{1}{T}\,\E\left[\int_0^{\t^y(0)}(x-y) \dd t\right] + \frac{1}{T}\,\E\left[\int_{\t^y(0)}^{\t^x(0)} Y_{x,y}(t) \dd t\right] \\
&= \frac{1}{T}\,\E[\t^y(0)](x-y) + \frac{1}{T}\,\E\left[\int_{\t^y(0)}^{\t^x(0)} Y_{x,y}(t) \dd t\right].
\end{align*}
By \eqref{eq:Yxy} and the Strong Markov property holding for L\'evy processes \cite{Asmussen2003}, observe that \\* $Y_{x,y}(t) \ed Y_{x-y,0}(\tau^y(0)+t)$, whereby
\begin{equation}
\frac{1}{T}\,\E\left[\int_{\t^y(0)}^{\t^x(0)} Y_{x,y}(t)\,\dd t\right] = \frac{1}{T}\,\E\left[\int_{0}^{\t^{x-y}(0)} Y_{x-y,0}(t) \dd t\right] = \Psi^{x-y,0}_T,
\end{equation}
which completes the proof.
\end{proof}

\subsection{Proof of Lemma 3}
\begin{proof}
Note that $Y^{z,0}(t)$ and $\tau^z(y)$ are intimately related. Namely, due to the fact that $X$ has no negative jumps 
\begin{equation}
\{ \tau^z(y) \leq t\} = \{Y_{z,0}(t) \leq y \}.
\end{equation}
In fact, $Y_{z,0}(\tau^z(y)) = y$, which implies that $\tau_z$ is a right inverse for $Y_{z,0}(t)$. Therefore, the following equality holds
\begin{equation}
\int_0^{\tau^z(0)} Y_{z,0}(t)\, \dd t = \int_0^z \tau^z(y)\, \dd y,
\end{equation}
which implies
\begin{equation}
\Psi^{z,0}_T = \frac{1}{T}\,\int_0^z \E[\tau^z(y)]\, \dd y = \frac{1}{T}\,\int_0^z \E[\tau^{z-y}(0)]\,\dd y =\frac{1}{T}\, \int_0^z \E[\tau^{y}(0)] \,\dd y.
\end{equation}
\end{proof}

\subsection{Proof of Corollary 1}
\begin{proof}
From \eqref{eq:invCharExp},
\begin{equation}\label{eq:corEq1}
\E[\hat{\t}^0(y)] = -\tfrac{\dd}{\dd u} \left. \E[\exp(-u\,\hat{\t}^0(y))]\right|_{u=0} = {-}y\left.\frac{\dd}{\dd u} \Upsilon^{-1}(u)\right|_{u=0}.
\end{equation}
Since $\Upsilon(u)$ is strictly increasing and $\Upsilon(0)=0$, we get $\Upsilon^{-1}(0)$ and
\begin{equation}
\left.\tfrac{\dd}{\dd u}\Upsilon^{-1}(u)\right|_{u=0} = \frac{1}{\Upsilon'(\Upsilon^{-1}(0))} = \{ \Upsilon'(0) \}^{-1}.
\end{equation}
Furthermore,
\begin{equation}
\Upsilon'(\th) = -a+ \s^2\th + \int_{-\infty}^0 (x\, e^{\th x} - x{\bf 1}_{[-1,0)}(x)) \hat{\nu}(\dd x) =  -a + \s^2\th - \int_0^\infty (y\, e^{-\th y} - y{\rm 1}_{(0,1]}(y)) \nu(\dd y).
\end{equation}
Thus, $\Upsilon'(0) = -\E[X(1)] = \mu-\l$ and $\E[\hat{\t}^0(y)] = -y/(\mu-\l)$. By \eqref{eq:transformedTau} and \eqref{eq:H(x,0)}, we deduce that
\begin{equation}
\Psi^{z,0}_T = \frac{1}{T}\, \int_0^z \E[\tau^z(0)] \,\dd y = \frac{1}{T}\, \int_0^z \E[\hat{\t}_0^y] \dd y = \frac{x^2}{2T(\mu-\l)}.
\end{equation}
For $x,y\geq 0$, we use Lemma 2 to conclude
\begin{equation}
\Psi^{x,y}_T =  \frac{-y(x-y)}{T(\mu-\l)} + \frac{-(x-y)^2}{2 T(\mu-\l)} = \frac{x^2-y^2}{2T(\mu-\l)}  .
\end{equation}
\end{proof}

\subsection{Proof of Proposition 1}

\begin{proof}
To derive the upper bound for $\Delta^{x,y}_T$, we apply the same coupling argument as in described in Section 3. Let us assume without loss of generality $x>y$. 
In this case, 
\begin{equation}
|\Delta^{x,y}_T| = \frac{1}{T} \int_T^\iy \E[Q^x(t)-Q^y(t)]\, \dd t \leq \frac{1}{T}\int_T^\iy \E[Q^x(t)-Q^0(t)]\, \dd t.
\end{equation}
By the decomposition in \eqref{eq:Yxy},
\begin{align*}
 \int_T^\infty \E[Q_x(t) - Q_0(t)]\, \dd t
	&= \int_T^\infty \E[(x+\inf_{s\leq t} X(s))\textbf{1}_{\{\tau^x(0)>t\}}] \, \dd t\\
	&= \int_T^\infty \int_0^x \P( x-u + \inf_{s\leq t}X(s) > 0) \, \dd u\, \dd t\\
	&= \int_T^\infty \int_0^x \P( \tau^{x-u}(0) > t ) \,\dd u\,\dd t \leq \int_T^\iy \int_0^x \frac{\E[\tau^{x-u}(0)^2]}{t^2}\,\dd u \,\dd t\\
	&=  \int_0^x \int_T^\iy \frac{\E[\tau^{x-u}(0)^2]}{t^2}\,\dd t \,\dd u
	= \int_0^x \frac{\E[\tau^{v}(0)^2]}{T}\,\dd v.
\end{align*}

We obtain $\E[\hat{\tau}_v^2]$ with the help of its Laplace transform in \eqref{eq:invCharExp}. Namely,
\begin{equation}
\E[\tau^v(0)^2] = \left.\tfrac{\dd^2}{\dd u^2}\E[\exp(-u \tau^v(0))]\right|_{u=0} = v^2\,\left(\left.\tfrac{\dd}{\dd u}\Upsilon^{-1}(u)\right|_{u=0}\right)^2 - v\left. \tfrac{\dd^2}{\dd u^2}\Upsilon^{-1}(u)\right|_{u=0}
\end{equation}
As in the previous subsection we have $\left.\tfrac{\dd}{\dd u}\Upsilon^{-1}(u)\right|_{u=0} = \mu-\l$, and 
\begin{equation}
\left.\tfrac{\dd^2}{\dd u^2}\Upsilon^{-1}(u)\right|_{u=0} = {-}\frac{\Upsilon''(\Upsilon^{-1}(0))}{\Upsilon'(\Upsilon^{-1}(0))^3} = {-}\frac{\Upsilon''(0)}{\Upsilon'(0)^3}.
\end{equation}
Since $\Upsilon'(0) = 1/(\mu-\l)$ and 
\begin{equation}
\Upsilon''(0) = \s^2 + \int_0^\infty x^2\,\nu(\dd x) = u_2,
\end{equation}
we conclude
\begin{equation}
\E[\tau^v(0)^2] = \frac{v^2}{(\m-\l)^2} + \frac{u_2v}{(\m-\l)^3},
\end{equation}
so that
\begin{equation}\label{eq:delta_upper}
|\Delta^{x,y}_T| \leq \frac{1}{T^2}\int_0^x \frac{v^2}{(\m-\l)^2} + \frac{u_2v}{(\m-\l)^3} \dd v = \frac{1}{T^2}\left(\frac{x^3}{3(\mu-\l)^2}+\frac{u_2 x^2}{2(\m-\l)^3}\right).
\end{equation}
For general $x,y\geq 0$, 
\begin{equation}
|\Delta^{x,y}_T| \leq \frac{1}{T^2}\left(\frac{\max(y,x)^3}{3(\mu-\l)^2}+\frac{u_2 \max(y,x)^2}{2(\m-\l)^3}\right).
\end{equation}
As a direct consequence,
\begin{equation}
|\Delta_T| \leq \frac{1}{T^2}\left(\frac{\E[\max(Q(0),Q_\m(\iy))^3]}{3(\mu-\l)^2}+\frac{u_2 \E[\max(Q(0),Q_\m(\iy))^2]}{2(\m-\l)^3}\right).
\end{equation}

\end{proof}

\section{Proofs of Section 4}
\subsection{Proof of Proposition 2}
In the proof of the proposition, we use the following auxiliary lemma, of which we include the proof for completeness.
\begin{lemma}\label{lemma:minimizerConvergence}
Consider the sequence of functions $f_n:\, [x_0,\infty) \to \mathbb{R}$ and let $f: [x_0,\infty) \to\mathbb{R}$ be the pointwise limit for some $x_0\in \mathbb{R}$. 
Assume $f$ and $f_n$ are strictly convex for all $n$. 
Furthermore, let $f(y) \to \infty$ for both $y\to x_0^+$ and $y\to \infty$.  
If $x_n$ and $x$ are the minimizers for $f_n$ and $f$, respectively, then $x_n\to x$ for $n\to\infty$.  
\end{lemma}
\begin{proof}
We start by showing that the sequence $x_n$ is bounded. Fix $u_l, u_r$ such that $x_0<u_l < x < u_r$. We claim that there exists a $N\in\mathbb{N}$ such that $x_n\in[u_l,u_r]$ for all $n \geq N$. First, we prove the upper bound on $x_n$. For any strictly convex function $h$ with minimizer $x_h$, the following statement holds true:
\begin{equation}\label{eq100}
x_h < u_r \quad \Leftrightarrow \quad h \text{ is strictly increasing at } u_r.
\end{equation}
The first implication follows from observing that $h(x^*) < h(y)$ for all $y> x^*$ and definition of convexity:
\[ 0<\frac{h(u_r)-h(x_h)}{u_r-x_h} \leq \frac{h(u_r+\d)-h(u_r)}{\d}, \]
for all $\d>0$. So that $h(u_r)<h(u_r+\d)$, i.e. $h$ is increasing at $u_r$. The converse follows immediately by observing that $h(u_r) < h(u_r+\d)$ for all $\d>0$, so that $x_h < u_r$. 
Next, we show that $f_n$ must be increasing at $u_r$ for $n$ sufficiently large. By pointwise convergence of $f_n$ we have
\[ \lim_{n\to\infty} [f_n(u_r+\d) - f_n(u_r)] = f(u_r+\d) - f(u_r).\]
Let $w_r:= f(u_r+\d) - f(u_r)>0$. Then 
\[ \exists N_r \in \mathbb{N}:\, \forall n\geq N_r:\, |[f_n(u_r+\d) - f_n(u_r)] - [f(u_r+\d)-f(u_r)] | < w_r/2.\]
Hence for $n\geq N_r$,
\[f(u_r+\d)-f(u_r) - w_r/2 <  f_n(u_r+\d) - f_n(u_r) < f(u_r+\d)-f(u_r) + w_r/2\]
\[\Rightarrow 0 < w_r/2 < f_n(u_r+\d) - f_n(u_r).
\]
Hence by \eqref{eq100}, $x_n < u_r$ for sufficiently large $n$. Similarly, we argue
\begin{equation}
x_h > u_l \quad \Leftrightarrow \quad h \text{ is strictly decreasing at } u_l,
\end{equation}
for any strictly convex function $h$ with minimizer $x_h$. Note that $x_h > u_l$ implies $h(x_h) - h(u_l) < 0$ and for all $\d>0$ we get by strict convexity
\[\frac{h(u_l)-h(u_l-\d)}{\d} < \frac{h(x_h)-h(u_l)}{x_h-u_l} < 0,\]
by which $h(u_l-\d)>h(u_l)$, i.e. $h$ is decreasing in $u_l$. Moreover, if $h$ is decreasing at $u_l$, then it is decreasing for all $y < u_l$, by arguments similar to the above. Therefore, $h(u_l-\d)> h(u_l)$ for all $\d>0$ and it must hold that $x_h>u_l$. Define $f(u_l) - f(u_l-\d) :=w_l < 0$, then again by pointwise convergence, we have that
\[ \exists N_l \in \mathbb{N}:\, \forall n\geq N_l:\, |[f_n(u_l) - f_n(u_l-\d)] - [f(u_l)-f(u_l-\d)] | < w_l,\]
whereupon
\[ f_n(u_l) - f_n(u_l-\d) < f(u_l) - f(u_l-\d) + w_l = 2w_l < 0.\]
Hence, for sufficiently large $n$, we also have $x_n > u_l$. Fix $N = \max\{N_l,N_r\}$, then for $n\geq N$, $x_n\in( u_l,u_r)$. That is, the sequence $x_n$ is bounded. Therefore, by the theorem of Bolzano-Weierstrass, $x_n$ has to have a convergent subsequence. That is, there exists a sequence $n_k$ such that $n_k \to\infty$ and $x_{n_k}\to a$ as $k\to \infty$ for some $a \in [u_l,u_r]$. 
We prove that every subsequence must converge to $x$ by contradiction. Suppose there exists a subsequence $n_k$ such that $x_{n_k}\to a\neq x$. Since, $x_n\in [u_l,u_r]$ for $n\geq N$, we may restrict our attention on the sequence of functions $\hat{f}_n:[u_l,u_r] \to \mathbb{R^+}$, consisting of the original function $f_n$ restricted to the domain $[u_l,u_r]$. To be precise $x_n = \arg\min_y f_n(y) = \arg\min_y \hat{f}_n(y)$ for $n\geq N$. Because $\hat{f}_n$ and $\hat{f}$ are bounded, we furthermore $\hat{f}_n \to \hat{f}$ uniformly.

Fix $\e>0$. By uniform convergence there exists an $K \in\mathbb{N}$ such that
\[  | \hat{f}_{n_k}( y ) - \hat{f}( y)| < \e /2,\quad \forall k\geq K_0,\ y \in[u_l,u_r].\]
Also, because $\hat{f}$ is convex, it is continuous, so that there exists a $\d := \d(\e)$ so that
\[ |z-y| < \d \quad \Rightarrow \quad |\hat{f}(z) - \hat{f}(y)| < \e/2.\]
Let $K_1$ be such that $|x_{n_k}-a| < \d$ for all $k\geq K_1$. Then for $k \geq K= \max\{K_0,k_1\}$ this implies. 
\begin{align*}
|f_{n_k}(x_{n_k}) - f(a)| &= |\hat{f}_{n_k}(x_{n_k}) - \hat{f}(a)| \\
&\leq  |\hat{f}_{n_k}(x_{n_k}) - \hat{f}(x_{n_k}) + | \hat{f}(x_{n_k}) - f(a)| < \e/2 + \e/2 = \e.
\end{align*} 
Hence we conclude $\lim_{k\to\infty} \hat{f}_{n_k}(x_{n_k}) = f(a)$. 
Therefore, 
\[ \limsup_{n\to \infty} f_n(x_n) \geq f(a) > f(x),\]
by minimality of $x$. However, $f_n(x_n) \leq f_n(x)$, which implies $\limsup_{n\to\infty} f_n(x_n) \leq \lim_{n\to\infty} f_n(x) = f(x)$, contradicting the strict inequality above. Hence we deduce $x=a$. Consequently, every subsequence of $x_n$ converges to $x$ and therefore $x_n\to x$ as $n\to \infty$.
\end{proof}

Since pointwise convergence is trivial, it remains to be proven that $\Pi_T(\mu)$ is strictly convex.
Since the term $\a\mu$ is convex, the strictness should come from the first term. 
Furthermore, observe that if a function $f_\mu(t)$ is convex for all $t\geq 0$, and strictly convex for all $t\geq\e$ for some $\e\in[0,T)$, i.e. for any $\mu_1,\mu_2>0$ and $a\in (0,1)$
\begin{equation}
a\, f_{\mu_1}(t) + (1-a) f_{\mu_2}(t) > f_{a\mu_1+(1-a)\mu_2}(t)
\end{equation}
then
\begin{equation}
a\int_0^T \, f_{\mu_1}(t)\, \dd t + (1-a)\int_0^T f_{\mu_2}(t) \dd t =
\int_0^T \, a f_{\mu_1}(t) + (1-a)f_{\mu_2}(t) \dd t 
\end{equation}
\begin{equation}
=  \int_0^\e a f_{\mu_1}(t) + (1-a)f_{\mu_2}(t) \dd t  + \int_\e^T \, a f_{\mu_1}(t) + (1-a)f_{\mu_2}(t) \dd t 
\end{equation}
\begin{equation}
> \int_0^\e  f_{a\mu_1+(1-a)\mu_2}(t) \, \dd t + \int_\e^T  f_{a\mu_1+(1-a)\mu_2}(t) \, \dd t.
= \int_0^T  f_{a\mu_1+(1-a)\mu_2}(t) \, \dd t.
\end{equation}

Hence, it suffices to prove the convexity of $\E[Q_\mu(t)]$ as a function of $\mu$ for all $t\geq 0$, and strict convexity for $t\geq \e$ for some $\e\in[0,T)$.

Recall the representation of the queue length process $Q_\mu(t)$, given server speed $\mu$ and initial state $Q(0)=x$:
\begin{align}
Q_\mu(t) &= U(t)-\mu t + \max\left\{ x , -\inf_{s\leq t} [ U(s)-\mu s] \right\}\\
&=
\left\{
\begin{array}{ll}
x+U(t)-\mu t, & \text{if } t<\tau(x,\mu),\\
U(t)-\mu t  -\inf_{s\leq t} [ U(s)-\mu s], & \text{if } t\geq \tau(x,\mu) ,
\end{array}\right.\label{eq:Qrep}
\end{align}
where
\begin{equation}
 \tau(x,\mu) := \inf\{ t \geq 0\,:\, x+U(t)-\mu t \leq 0\}
 \end{equation} 
and $U(t)$ is a spectrally positive L\'evy process.
Fix $\mu_1, \mu_2>0$ and $a\in(0,1)$. Define $\mu_3 := a\mu_1+(1-a)\mu_2$, and
\begin{equation}
D(t) := a Q_{\mu_1}(t) + (1-a) Q_{\mu_2}(t) - Q_{\mu_3}(t).
\end{equation}
In order to prove strict convexity we have show that $D(t) \geq 0$ for all $t\geq 0$, thereby implying $\E D(t) \geq 0$, i.e. convexity, for all $t\geq 0$, and $D(t)>0$ with positive probability for $t\in[\e,T]$, for some $\e \in[0,T)$. 

We distinguish two cases: $x>0$ and $x=0$. \\
\\*
\textbf{Case $x>0$.}
We start by noticing that if $Q_{\mu_1}$, $Q_{\mu_2}$ and $Q_{\mu_3}$ experience the same input process $U(t)$, then by absence of negative jumps in $U(t)$, it holds that
\begin{equation}\label{eq:stochDom}
\tau(x,\mu_2) < \tau(x,\mu_3) < \tau(x,\mu_1).
\end{equation}
We use shorthand notation 
\begin{equation}
I_k(t) := \inf_{0\leq s<\leq t}[U(s)-\mu_k s],
\end{equation}
for $k=1,2,3$.
Using representation \eqref{eq:Qrep} of the workload process, we obtain 
\begin{equation}
D(t) = \left\{
\begin{array}{ll}
0,  
& \text{if } t < \tau(x,\mu_2),\\
-(1-a)\left(x+I_2(t) \right), 
& \text{if } \tau(x,\mu_2) \leq t < \tau(x,\mu_3),\\
a x - (1-a)I_2(t) + I_3(t),
& \text{if } \tau(x,\mu_3) \leq t < \tau(x,\mu_1),\\
- a I_1(t) - (1-a) I_2(t)
+ I_3(t),
& \text{if } t \geq \tau(x,\mu_1).
\end{array}
\right.
\end{equation}
This partition of allows us to spot when strict convexity can occur.
Note that by definition $t \geq \tau(x,\mu_2)$, $\inf_{0\leq s<\leq t}[U(s)-\mu_2s] \leq x)=I_2(t)\leq x$, so that $D(t)\geq 0$ if $\tau(x,\mu_2) \leq t < \tau(x,\mu_3)$. 
Moreover, by subadditivity of the infimum, 
\[
I_3=\inf_{0\leq s<\leq t}[U(s)-\mu_3s] = \inf_{0\leq s<\leq t}[a(U(s)-\mu_1s)+(1-a)(U(s)-\mu_2s)] 
\]
\begin{equation}
\geq a \inf_{0\leq s<\leq t}[U(s)-\mu_2s] \leq x) + (1-a) \inf_{0\leq s<\leq t}[U(s)-\mu_2s] \leq x) = a I_1(t) + (1-a) I_2(t),
\end{equation}
and hence $D(t)\geq 0$ for $t \geq \tau(x,\mu_1)$.
Using the same argument, we deduce
\begin{equation}
ax - (1-a)I_2(t) + I_3(t) \geq a x - (1-a) I_2(t) + a I_1(t) + (1-a) I_2(t) = a(x + I_1(t)).
\end{equation}
In particular for $t < \tau(x,\mu_1)$, this value is strictly positive. 
As a result, $D(t)\geq 0$ for all $t\geq 0$.
On top of that $D(t) > 0$ for $t\in[\tau(x,\mu_3),\tau(x,\mu_1))$.
Accordingly, the latter implies strict positivity of $\E D(t)$, and therefore strict convexity of $\E Q_\mu(t)$, if the event $\{\tau(x,\mu_3)\leq t< \tau(x,\mu_1)\}$ occurs with positive probability. 
That is, 
\begin{align}
\P(D(t)>0) &\geq \P\left( a(x+I_1(t))\textbf{1}_{\{\tau(x,\mu_3)\leq t < \tau(x,\mu_1)\}}  > 0 \right)\nonumber\\
&= \P\left( x+ I_1(t) > 0 , \tau(x,\mu_3)\leq t < \tau(x,\mu_1)\right)\nonumber\\
&= \P\left( x+ I_1(t) > 0 | \tau(x,\mu_3)\leq t < \tau(x,\mu_1)\right)\P\left(\tau(x,\mu_3)\leq t < \tau(x,\mu_1)\right)\nonumber\\
&= \P\left(\tau(x,\mu_3)\leq t < \tau(x,\mu_1)\right) = \P(\tau(x,\mu_3)\leq t) - \P(\tau(x,\mu_1) \leq t) > 0, \label{eq:strictConv}
\end{align}
by the stochastic dominance in \eqref{eq:stochDom}. To ensure the strict inequality in \eqref{eq:strictConv} we have to enforce the condition
\begin{equation}\label{eq:condition}
\P(\tau(x,\mu_1)<T) > 0.
\end{equation}
\textit{Remark.}
An example illustrating the need for this condition is the case in which $U(t)$ is a compound Poisson process and $T < x/\mu_2 < x/\mu_1$. Then 
\[Q_{\mu_k}(t) = x + U(t) - \mu_k t,\]
for all $t\in[0,T]$, since $U(t)\geq 0$ and therefore $\tau(x,\mu_1) > T$. Consequently, for all $a\in(0,1)$, 
\[ a\,Q_{\mu_1} + (1-a)\,Q_{\mu_2}(t) = Q_{\mu_3}(t),\]
proving only convexity of $\E Q_{\mu}(t)$ and subsequently $\int_0^T \E[Q_\mu(t)]\,\dd t$. In case $\sigma>0$, the probability in \eqref{eq:condition} is necessary positive.
\\*

\textbf{The case $x=0$.}
By the fact that $\tau(0,\mu) = 0$ for all $\mu>0$. Proving that $D(t)>0$ for in the case $x=0$ reduces to showing that the probability of 
\begin{equation}
D(t) = a I_1(t) + (1-a) I_2(t) - I_3(t)>0
\end{equation}
 happening is positive for all $t>0$. Define 
\begin{equation}
 t_0 := \inf\{ t > 0\, :\, U(t) > 0 \},
 \end{equation} 
and
\begin{equation}
\tilde{\tau}(\mu) := \inf\{ t > t_0\,: U(t) - \mu t \leq 0\}.
\end{equation}
We note that $t_0$ as defined above, also defines the epoch of the start of a new excursion of the reflection $Q_\mu$ for all $\mu>0$. Namely,
\[U(s) \leq 0  \quad \Rightarrow\quad  U(s) - \mu s \leq -\mu s  \qquad \text{for all }0\leq s< t_0\]
\[\Rightarrow \inf_{0\leq s < t_0} [U(s)-\mu s] \leq -\mu t_0 \quad 
\Rightarrow U(t_0) - \mu t_0 -  \inf_{0\leq s < t_0} [U(s)-\mu s] \geq U(t_0) > 0\]
Then $Q_\mu(t_0-) = 0$ for all $\mu>0$. 
By the virtue of the Strong Markov Property, not that $Q_\mu(t_0+t) \ed Q_\mu(t)$. 
Hence we assume without loss of generality $t_0=0$. 
Again, we have a stochastic dominance relation similar to \eqref{eq:stochDom}:
\begin{equation}
\tilde{\tau}(\mu_2) < \tilde{\tau}(\mu_3) < \tilde{\tau}(\mu_1),
\end{equation}
for all $\mu_1<\mu_3<\mu_2$. 
Then 

\begin{equation}
D(t) \ed \left\{
\begin{array}{ll}
0,  
& \text{if } t < \tilde{\tau}(\mu_2),\\
-(1-a)I_2(t), 
& \text{if } \tilde{\tau}(\mu_2) \leq t < \tilde{\tau}(\mu_3),\\
(1-a)I_2(t) + I_3(t),
& \text{if } \tilde{\tau}\mu_3) \leq t < \tilde{\tau}\mu_1),\\
- a I_1(t) - (1-a) I_2(t)
+ I_3(t),
& \text{if } t \geq \tilde{\tau}\mu_1).
\end{array}
\right.
\end{equation}

Clearly, $D(t)\geq 0$ for all $t\geq 0$ and 
\begin{equation}
-(1-a)I_2(t) + I_3(t) \geq a I_1(t) > 0,
\end{equation}
for $\tilde{\tau}(\mu_3) \leq t < \tau(\mu_1)$. 
Hence, in a similar manner to \eqref{eq:strictConv},
\begin{align}
\P(D(t)>0) &\geq \P\left( aI_1(t)\textbf{1}_{\{\tilde{\tau}(\mu_3) \leq t < \tau(\mu_1)\}}  > 0 \right)\nonumber\\
&= \P\left( I_1(t) > 0 , \tilde{\tau}(\mu_3) \leq t < \tau(\mu_1)\right)\nonumber\\
&= \P\left( I_1(t) > 0 | \tilde{\tau}(\mu_3) \leq t < \tau(\mu_1)\right)\P\left(\tilde{\tau}(\mu_3) \leq t < \tau(\mu_1)\right)\nonumber\\
&= \P\left(\tilde{\tau}(\mu_3) \leq t < \tilde{\tau}(\mu_1)\right) = \P(\tilde{\tau}(\mu_3)\leq t) - \P(\tilde{\tau}(x,\mu_1) \leq t) > 0, \label{eq:strictConv2}
\end{align}
The last inequality is satisfied it $\P(\tilde{\tau}(\mu_1) < T) >0$, which is  equivalent to $\P( U(T) - \mu T \leq 0 ) >0$, a condition that is clearly true for all our choice of $U$. 
In conclusion, for $x=0$, $\E D(t) >0$ and therefore $\E Q_\mu(t)$ is a strictly convex function of $\mu$.

\subsection{Proof of Proposition 3}
\begin{proof}
Note that $\Pi_\infty$ is a smooth function. 
By the first optimality condition $\Pi_T'(\muT)= 0$ and $\Pi_\infty'(\mui) = 0$. 
Furthermore, by definition $\Psi_T(\mu) = O(1/T)$ and $\Delta_T(\mu) = O(1/T^2)$.
Hence,
\begin{align*}
0=\Pi_T'(\muT)  
&= \Pi_\infty'(\muT) + \Psi_T'(\muT) + O(1/T^2)\\
&= \Pi_\infty'(\mui) + \Psi_T'(\mui) + (\muT-\mui)\left[ \Pi_\infty''(\mui) + \Psi_T''(\mui) \right] \\
&\qquad + \frac{1}{2}(\mu_T-\mui)^2\left[\Pi_T'''(\xi)+\Psi_T'''(\xi) \right] + O(1/T^2)\\
&= \Psi_T'(\mui) + (\muT-\mui)\left[ \Pi_\infty''(\mui) + \Psi_T''(\mui) \right] \\
&\qquad + \frac{1}{2}(\mu_T-\mui)^2\left[\Pi'''(\xi)+\Psi_T'''(\xi)\right] + O(1/T^2).
\end{align*}
for some $\xi \in [\muT,\mui]$. Rearranging this gives
\begin{align*}
\muT-\mui &= \frac{-\Psi_T'(\mui)}{\Pi_\infty''(\mui)+\Psi_T''(\mui) + \frac{1}{2}(\muT-\mui)(\Pi_\infty'''(\muT)+\Psi_T'''(\xi))} + O(1/T)\\
&= {-}\frac{\Psi_T'(\mui)}{\Pi_\iy''(\mui)} \left[1 - \frac{\Psi_T''(\mu)}{\Pi_\infty''(\mui)} - \frac{1}{2}(\muT-\mu_\infty)\frac{\Pi_\infty'''(\mui)+\Psi_T'''(\mui)}{\Pi_\infty''(\mui)}\right] + O(1/T)\\
&= {-}\frac{\Psi_T'(\mui)}{\Pi_\iy''(\mui)} [1 + o(1)]
\end{align*}
for $T\to\infty$, since both $\mu_T - \mu_\infty$ and $\Psi_T''(\mui)$ are $o(1)$. 
Let
\begin{equation}
\mu_\bullet := \lim_{T\to\iy} \frac{T \Psi_T'(\mui)}{\Pi_\iy''(\mui)}.
\end{equation}
By \eqref{eq:mainResult} we have
\begin{equation}
T \Psi'(\mu) = {-} \frac{\E[Q(0)^2]}{2(\mu-\l)^2} + \frac{\l u_3}{3(\mu-\l)^3} + \frac{3\l^2u_2^2}{4(\mu-\l)^4}.
\end{equation}
Together with 
\begin{equation}
\Pi_\iy''(\mu) = \frac{\l u_2}{(\mu-\l)^3}
\end{equation}
and \eqref{eq:muInf} we obtain the expression for $\mu_\bullet$ in \eqref{eq:muBullet}.
\end{proof}
\subsection{Proof of Proposition 4}\label{sec:proofProp4}
\begin{proof}
We upper bound the optimality gap by using the decomposition in \eqref{eq:decomposition}.
\begin{align}
|\Pi_T(\mui) - \Pi_T^\star| &= \left|\hat{\Pi}_T(\mu_\infty) + \Delta_T(\mui) - \hat{\Pi}_T(\muT) - \Delta_T(\muT)\right|\nonumber\\
&\leq |\hat{\Pi}_T(\mui) - \hat{\Pi}_T(\muT)| + |\Delta_T(\mui)| + |\Delta_T(\muT)|\nonumber\\
&= |\hat{\Pi}_T(\mui) - \hat{\Pi}_T(\muT)| + O(1/T^2),
\end{align}
since $\Delta_T(\mu) = O(1/T^2)$ by Proposition \ref{prop:truncation_error}. next, we find an upper bound for $|\hat{\Pi}_T(x) - \hat{\Pi}_T(y)|$ in terms of the difference between $x$ and $y$. 
For simplicity, denote $\hat{x} = x - \l$ and $\hat{y} = y-\l$, implying $\hat{x}-\hat{y}=x-y$. Then using the expression of in \eqref{eq:mainResult} we get
\begin{align*}
|\hat{\Pi}_T(\mui) - \hat{\Pi}_T(\muT)| &= 
\left| \a(\hat{x}-\hat{y})
+\left(\frac{\l u_2}{2} + \frac{\E[Q(0)^2]}{2T}\right)\left(\frac{1}{\hx}-\frac{1}{\hy}\right) \right. \\
& \qquad \left. -\frac{\l^2 u_2^2}{4T}\left(\frac{1}{\hx^3}-\frac{1}{\hy^3}\right) 
-\frac{\l u_3}{6T}\left(\frac{1}{\hx^2} - \frac{1}{\hy^2}\right)
\right|.
\end{align*}
Furthermore, we have 
\begin{align*}
\frac 1 \hx - \frac 1 \hy &= -\frac{\hx-\hy}{\hy^2} + \frac{(\hx-\hy)^2}{\hy^3} + O\left((x-y)^3\right),\\
\frac 1 {\hx^2} - \frac 1 {\hy^2} &= -\frac{2(\hx-\hy)}{\hy^3} + \frac{3(\hx-\hy)^2}{\hy^4} + O\left((x-y)^3\right),\\
\frac 1 {\hx^3} - \frac  1 {\hy^3} &= -\frac{3(\hx-\hy)}{\hy^4} + \frac{6(\hx-\hy)^2}{\hy^5} + O\left((x-y)^3\right),\\
\end{align*}
Substituting these yields 
\begin{align*}
|\hat{\Pi}_T(x) - \hat{\Pi}_T(y)| &= \left|(x-y)\left[ \a - \frac{\l u_2}{2 \hy^2} + \frac{1}{2T\hy^2}\left(\E[Q(0)^2] + \frac{3\l^2 u_2^2}{2\hy^2} + \frac{2\l u_3}{3 \hy}\right)\right]\right. \\
&\qquad \left. - (x-y)^2\left[ \frac{\l u_2}{2 \hy^3} + \frac{1}{2T\hy^3}\left(\E[Q(0)^2] - \frac{3\l^2 u_2^2}{\hy^2} - \frac{\l u_3}{\hy}\right)\right]\right| + O\left((x-y)^3\right).
\end{align*}
Given that $\muT = \mui + \mu_\bullet/T + o(1/T)$, we find
\begin{align*}
|\hat{\Pi}_T(\mui) - \hat{\Pi}_T(\muT)| &= \frac{|\mu_\bullet|}{T}\left(\a - \frac{\l u_2}{2(\mui-\l u_1)^2}\right) + O(1/T^2)\\
&= \frac{|\mu_\bullet|}{T}\left(\a - \frac{\l u_2}{2(\sqrt{\l u_2/2\a})^2}\right) + O(1/T^2) = O(1/T^2),
\end{align*}
which concludes the proof.
\end{proof}
\end{document}